\documentclass[11pt]{article}
\usepackage[utf8]{inputenc}
\usepackage[english]{babel}
\usepackage{epsfig,enumerate,amsmath,amsfonts,amsbsy,amssymb,amsthm,mathrsfs,lineno,ifpdf}
\usepackage{indentfirst,relsize,capt-of}
\usepackage[mathscr]{euscript}
\usepackage[numbers]{natbib}
\usepackage{setspace,graphicx}
\usepackage[normalem]{ulem}
\usepackage{latexsym}
\usepackage[usenames,dvipsnames]{pstricks}
\usepackage[color=blue!20,textsize=small,textwidth=2.2cm]{todonotes}
\usepackage{array}
\usepackage[inline]{enumitem}

\usepackage{pst-grad} 
\usepackage{pst-plot} 
\usepackage{tikz,pgf} 
\usetikzlibrary{positioning,shapes,shadows,arrows} 
\usepackage[lined,ruled,linesnumbered]{algorithm2e}
\usepackage[margin=1.2in]{geometry}
\newtheorem{theorem}{Theorem}
\newtheorem{lemma}{Lemma}

\newtheorem{obs}{Observation}

\newtheorem{conj}{Conjecture}

\newtheorem{claim}{Claim}

\usepackage{enumitem}
\usepackage{etoolbox}

\let\oldenumerate\enumerate
\renewcommand{\enumerate}{
	\oldenumerate
	\setlength{\itemsep}{1.5pt}
	\setlength{\parskip}{0pt}
	\setlength{\parsep}{0pt}
}

\numberwithin{equation}{section}

\usepackage{hyperref, url}
\usepackage{graphics}
\usepackage[most]{tcolorbox}
\hypersetup{
	colorlinks=true,
	linkcolor=blue,
	pdfpagemode=FullScreen,
}
\newtcolorbox{mytextbox}[1][]{%
	sharp corners,
	enhanced,
	colback=white,
	height=10cm,
	attach title to upper,
	#1
}

\graphicspath{{images/}{../images/}}

\begin{document}
	\title{List recoloring of planar graphs}

\author{L. Sunil Chandran$^{\rm a}$, Uttam K. Gupta$^{\rm b}$, Dinabandhu Pradhan$^{\rm b,}$\thanks{Corresponding author.}  \\ \\
		$^{\rm a}$Department of Computer Science and Automation\\Indian Institute of Science, Bangalore, India\\
		\\		
		$^{\rm b}$Department of Mathematics \& Computing\\ Indian Institute of Technology (ISM), Dhanbad\\	
		\\
		\small \tt Email: sunil@iisc.ac.in; ukumargpt@gmail.com; dina@iitism.ac.in}
	
	\date{}
	\maketitle

\begin{abstract}

A list assignment $L$ of a graph $G$ is a function that assigns to every vertex $v$ of $G$ a set $L(v)$ of colors. A proper coloring $\alpha$ of $G$ is called an $L$-coloring of $G$ if $\alpha(v)\in L(v)$ for every $v\in V(G)$. For a list assignment $L$ of $G$, the $L$-recoloring graph $\mathcal{G}(G,L)$ of $G$ is a graph whose vertices correspond to the $L$-colorings of $G$ and two vertices of $\mathcal{G}(G,L)$ are adjacent if their corresponding $L$-colorings differ at exactly one vertex of $G$. A $d$-face in a plane graph is a face of length $d$. Dvo\v{r}\'ak and Feghali conjectured for a planar graph $G$ and a list assignment $L$ of $G$, that: (i) If $|L(v)|\geq 10$ for every $v\in V(G)$, then the diameter of $\mathcal{G}(G,L)$ is $O(|V(G)|)$. (ii) If $G$ is triangle-free and $|L(v)|\geq 7$ for every $v\in V(G)$, then the diameter of $\mathcal{G}(G,L)$ is $O(|V(G)|)$. In a recent paper, Cranston (European J. Combin. (2022)) has proved (ii). In this paper, we prove the following results. Let $G$ be a plane graph and $L$ be a list assignment of $G$. 
\begin{itemize}
\item If for every $3$-face of $G$, there are at most two $3$-faces adjacent to it and $|L(v)|\geq 10$ for every $v\in V(G)$, then the diameter of $\mathcal{G}(G,L)$ is at most $190|V(G)|$.
\item If for every $3$-face of $G$, there is at most one $3$-face adjacent to it and $|L(v)|\geq 9$ for every $v\in V(G)$, then the diameter of $\mathcal{G}(G,L)$ is at most $13|V(G)|$.
\item If the faces adjacent to any $3$-face have length at least $6$ and $|L(v)|\geq 7$ for every $v\in V(G)$, then the diameter of $\mathcal{G}(G,L)$ is at most $242|V(G)|$. This result strengthens the Cranston's result on (ii).
\end{itemize}

  As an additional result, we show that if the independence number of a $k$-colorable graph $G$ is at most $p$ and $L$ is a list assignment of $G$ such that $L(v)=\{1,2,\ldots,\ell\}$ for every $v\in V(G)$, where $\ell\geq \big\lfloor \frac{p\cdot k}{2}\big\rfloor+1$, then the diameter of $\mathcal{G}(G,\ell)$ is at most $4|V(G)|$. We further show that if $\ell<\big\lfloor \frac{p\cdot k}{2}\big\rfloor+1$, then $\mathcal{G}(G,L)$ can be a disconnected graph. 
\end{abstract}
	
\section{Introduction}

All the graphs considered in this paper are finite, simple, and undirected. For a graph $G$, we use $V(G)$ and $E(G)$ to denote the vertex set and the edge set of $G$, respectively. 
Consider a problem with several feasible solutions. The \emph{reconfiguration problems} deal with transforming one feasible solution of the considered problem to another by taking small steps. Each step of transformation is taken over the feasible solutions. The reconfiguration problems have applications in statistical physics and optimization (see \cite{mohar,nishimura,heuvel}). 

When the reference problem is graph coloring, the reconfiguration problems are generally known as \emph{recoloring problems}. A \emph{(proper) $k$-coloring} of a graph $G$ is an assignment of colors from the set $\{1,2,\ldots,k\}$ of colors to the vertices of $G$ such that the adjacent vertices do not receive the same color. A graph $G$ is \emph{$k$-colorable} if there exists a $k$-coloring of $G$. For a $k$-colorable graph $G$, the \emph{$k$-recoloring graph} $\mathcal{G}(G,k)$ is a graph whose vertices correspond to the $k$-colorings of $G$ and two vertices of $\mathcal{G}(G,k)$ are adjacent if their corresponding $k$-colorings of $G$ differ at exactly one vertex of $G$. We use diam$(G)$ to denote the diameter of a graph $G$. The \emph{$k$-recoloring diameter} of a $k$-colorable graph $G$ is diam$(\mathcal{G}(G,k))$. We simply write recoloring diameter instead of $k$-recoloring diameter if the context of the $k$-recoloring graph is clear.
		
Two immediate questions about reconfiguration problems are the following: (Q1) Can we always reach from one feasible solution of the given problem to another? (Q2) What is the minimum number of steps required if the answer to Q1 is affirmative? These questions are important, for example in random sampling: an affirmative answer to Q1 ensures the ergodicity of the Markov chain. The answer to Q2 provides a lower bound on the mixing time of the underlying Markov chain (see~\cite{chen,frieze}). In the recoloring problems, Q1 and Q2 are equivalent to the questions whether (i) the concerned recoloring graph is connected,  and (ii) if it is connected, then what is the best upper bound on the recoloring diameter. For every $k\geq 2$, there exists a $k$-colorable graph $G$ (for example, consider the complete graph $K_k$ on $k$ vertices) with two $k$-colorings $c_1$ and $c_2$ such that there is no path between the two vertices corresponding to $c_1$ and $c_2$ in $\mathcal{G}(G,k)$. In case of the complete graph $K_k$, it can be seen that $\mathcal{G}(K_k,k)$ consists of only isolated vertices. On the other hand, for every $k\geq 4$, there exists some $k$-colorable graph $G$ such that $\mathcal{G}(G,k)$ may be disconnected, but some of the connected components of $\mathcal{G}(G,k)$ have diameter superpolynomial in $|V(G)|$~\cite{bonsma}. In view of these observations, the study of upper bounds on recoloring diameter of graphs in the most general setting does not make sense.

A graph $G$ is $d$-degenerate if for every subgraph $H$ of $G$, there exists a vertex $v\in V(H)$ such that $d_H(v)\leq d$, where $d_H(v)$ is the degree of $v$ in $H$. If we restrict ourselves to the class of $d$-degenerate graphs, some upper bounds on the $k$-recoloring diameter can be obtained usually for larger values of $k$.  Cereceda et al.~\cite{cerecedadisc} and Dyer et al.~\cite{dyer} independently proved that if $G$ is a $d$-degenerate graph, then $\mathcal{G}(G,k)$ is connected for every $k\geq d+2$. The bound $k\geq d+2$ is best possible since for $k=d+1$, the complete graph $K_{d+1}$ is $d$-degenerate and $\mathcal{G}(K_{d+1},d+1)$ consists of only isolated vertices. Cereceda et al.~\cite{cerecedadisc} proved an exponential upper bound on diam$(\mathcal{G}(G,k))$ in terms of $|V(G)|$; however, Cereceda~\cite{cerecedaconjec} conjectured the following.

\begin{conj}[Cereceda's conjecture]\label{cerecedaconjec}
 Let $G$ be a $d$-degenerate graph and $k\geq d+2$. Then diam$(\mathcal{G}(G,k))$ is $O(|V(G)|^2)$.
\end{conj} 

Note that in the case of chordal graphs, Conjecture~\ref{cerecedaconjec} is equivalent to the statement that for every $k\geq \chi(G)+1$, diam($\mathcal{G}(G,k)$) is $O(|V(G)|^2)$ since degeneracy of a chordal graph $G$ is equal to $\chi(G)-1$. This statement is proved in~\cite{bonamychordal}. 

Though Conjecture~\ref{cerecedaconjec} is still open, researchers have already proved some weaker versions of the conjecture. Bousquet and Heinrich~\cite{bousquetpol} proved that if $G$ is a $d$-degenerate graph and $k\geq d+2$, then  diam$(\mathcal{G}(G,k))=O(|V(G)|^{d+1})$. They also proved that if $k\geq \frac{3}{2}(d+1)$, then diam$(\mathcal{G}(G,k))=O(|V(G)|^2)$. For every $k=d+2$, Cereceda et al.~\cite{cerecedaconjec} had proved that there exists a family $\mathcal{F}$ of $d$-degenerate graphs such that for every $G\in \mathcal{F}$, diam($\mathcal{G}(G, k)$) is $\Omega(|V(G)|^2)$. This shows the tightness of Conjecture~\ref{cerecedaconjec} for $k=d+2$. Bartier et al.~\cite{bartierarxiv} conjectured that if $G$ is a $d$-degenerate graph and $k\geq d+3$, then diam$(\mathcal{G}(G,k))=O(|V(G|)$. Bousquet and Perarnau~\cite{bousquet2d+2} proved the following result.

\begin{lemma}[\cite{bousquet2d+2}]\label{linear} 
If $G$ is a $d$-degenerate graph and $k\geq 2d+2$, then diam$(\mathcal{G}(G,k))\leq (d+1)|V(G)|$.
\end{lemma}

One of the most celebrated class of graphs of constant degeneracy is the class of planar graphs: it is well known that the planar graphs are $5$-degenerate and triangle-free planar graphs are $3$-degenerate. Bousquet and Heinrich~\cite{bousquetpol} proved that if $G$ is a planar bipartite graph, then diam$(\mathcal{G}(G,5))=O(|V(G)|^2)$ confirming the Cereceda's conjecture for the class of planar bipartite graphs. Recently, Cranston and Mahmoud~\cite{cranston22} strengthened this result by proving that if $G$ is a planar graph with no $3$-cycles and no $5$-cycles, then diam$(\mathcal{G}(G,5))=O(|V(G)|^2)$.  The \emph{girth} of a graph $G$ is the length of a smallest cycle in $G$. Bartier et al.~\cite{bartierarxiv} proved that if $G$ is a planar graph of girth at least $6$, then diam$(\mathcal{G}(G,5))=O(|V(G)|)$, and if $G$ is a planar graph of girth at least $5$, then diam$(\mathcal{G}(G,4))<\infty$.

 \begin{table}[h]
 	\relsize{-1}{
 \begin{center}

\begin{tabular}{|>{\centering\arraybackslash}p{2.6cm}|>{\centering\arraybackslash}p{.8cm}|>{\centering\arraybackslash}p{2.1cm}|>{\centering\arraybackslash}p{2.1cm}|>{\centering\arraybackslash}p{.8cm}|>{\centering\arraybackslash}p{2.1cm}|>{\centering\arraybackslash}p{2.1cm}|>{\centering\arraybackslash}p{.8cm}|}
	
    \hline
    \textbf{Class/No. of Colors} & \textbf{4} & \textbf{5} & \textbf{6} & \textbf{7} & \textbf{8} & \textbf{9} & \textbf{10}   \\
    \hline 
    Planar graphs of girth at least~$6$ & $O(n^3)$ \cite{bousquetpol} & $O(n)$ \cite{bartierarxiv} & $O(n)$ \cite{bousquet2d+2} & $O(n)$ \cite{bousquet2d+2} & $O(n) $\cite{bousquet2d+2} & $O(n)$ \cite{bousquet2d+2} & $O(n)$ \cite{bousquet2d+2}\\  
    \hline
    Planar graphs of girth at least~$5$ & $<\infty$ & $O(n(\log(n))^2)$ \cite{mad} & $O(n(\log(n))^2)$ \cite{mad} & $O(n)$ \cite{dvorak} & $O(n)$ \cite{bousquet2d+2} & $O(n)$ \cite{bousquet2d+2} & $O(n)$ \cite{bousquet2d+2}\\ 
    \hline
    Planar bipartite graphs & $\infty$ & $O(n^2)$ \cite{bousquetpol} & $O(n(\log(n))^3)$ \cite{mad} & $O(n)$ \cite{dvorak} & $O(n)$ \cite{bousquet2d+2} & $O(n)$ \cite{bousquet2d+2} & $O(n)$ \cite{bousquet2d+2}\\
    \hline
    Planar graph with no 3-cycles and no 5-cycles & $\infty$ & $O(n^2)$ \cite{cranston22} & $O(n(\log(n))^3)$ \cite{mad} & $O(n)$ \cite{dvorak} & $O(n)$ \cite{bousquet2d+2} & $O(n)$ \cite{bousquet2d+2} & $O(n)$ \cite{bousquet2d+2} \\
    \hline
    Triangle-free planar graph  & $\infty$ & $O(n^4)$ \cite{bousquetpol} & $O(n(\log(n))^3)$ \cite{mad} & $O(n)$ \cite{dvorak} & $O(n)$ \cite{bousquet2d+2} & $O(n)$ \cite{bousquet2d+2} & $O(n)$ \cite{bousquet2d+2} \\
    \hline
     Planar graph & $\infty$ & $\infty$ & $\infty$ & $O(n^6)$ \cite{bousquetpol} & $O(n(\log(n))^5)$ \cite{mad} & $O(n(\log(n))^5)$ \cite{mad} & $O(n)$ \cite{dvorak}  \\
    \hline

\end{tabular}
\caption{Best known upper bounds on the recoloring diameters of the graphs on $n$ vertices in the corresponding classes.}
\label{maintable}
 \end{center}}
 \end{table}

Since every planar graph $G$ is $5$-degenerate, by Lemma~\ref{linear}, diam$(\mathcal{G}(G,k))\leq 6|V(G)|$ if $k\geq 12$. Again since every triangle-free planar graph $G$ is $3$-degenerate, by Lemma~\ref{linear}, diam$(\mathcal{G}(G,k))\leq 4|V(G)|$ if $k\geq 8$. Dvo\v{r}\'{a}k and Feghali~\cite{dvorak} improved these bounds by showing that for a planar graph $G$, diam$(\mathcal{G}(G,10))\leq 8|V(G)|$, and for a triangle-free planar graph $G$, diam$(\mathcal{G}(G,7))\leq 7|V(G)|$. We refer to Table~\ref{maintable} for the known upper bounds on the recoloring diameter of the graphs in some well known classes of graphs to the best of our knowledge. Dvo\v{r}\'{a}k and Feghali~\cite{dvorak} conjectured that their results on planar graphs and triangle-free planar graphs can be generalized to list coloring. 

\noindent\textbf{List version:}
 A \emph{list assignment} $L$ of a graph $G$ is a function that assigns to every vertex of $G$ a set $L(v)$ of colors. A proper coloring $\alpha$ is called an \emph{$L$-coloring} of $G$ if $\alpha(v)\in L(v)$ for every $v\in V(G)$. For a list assignment $L$ of $G$, the \emph{$L$-recoloring graph} $\mathcal{G}(G,L)$ of $G$ is a graph whose vertices correspond to the $L$-colorings of $G$ and two vertices of $\mathcal{G}(G,L)$ are adjacent if their corresponding $L$-colorings differ at exactly one vertex of~$G$. 

\begin{conj}[\cite{dvorak}]\label{mainconj}
Let $G$ be a planar graph and $L$ be a list assignment of $G$.
\begin{enumerate}
\item If $|L(v)|\geq 10$ for every $v\in V(G)$, then  diam$(\mathcal{G}(G,L))=O(|V(G)|)$.
\item If $G$ is triangle-free and $|L(v)|\geq 7$ for every $v\in V(G)$, then  diam$(\mathcal{G}(G,L))=O(|V(G)|)$.
\end{enumerate} 
\end{conj}

Recently, Cranston~\cite{cranstonEJC} settled Conjecture~\ref{mainconj}(b) by proving that if $G$ is a triangle-free planar graph and $L$ is a list assignment of $G$ such that $|L(v)|\geq 7$ for every $v\in V(G)$, then diam$(\mathcal{G}(G,L))\leq 30(|V(G)|)$. Bartier et al.~\cite{bartierarxiv} proved a weakening of Conjecture~\ref{mainconj}(a) by showing that if $G$ is a planar graph and $L$ is a list assignment of $G$ such that $|L(v)|\geq 11$ for every $v\in V(G)$, then diam$(\mathcal{G}(G,L))\leq 100(|V(G)|)$. 

\subsection{Our results}

 A \emph{plane embedding} of a planar graph $G$ is an embedding of $G$ in a plane such that the edges of $G$ do not cross each other except at the endpoints. A planar graph with a plane embedding is called a \emph{plane} graph. A face $f$ of a plane graph $G$ is a $k$-face (respectively, $k^+$-face/$k^-$-face), if the length of $f$ is $k$ (respectively, at least $k$/ at most $k$). Let $\mathcal{G}_1=\{G\mid$ $G$ is a planar graph and $G$ has a plane embedding such that for every $3$-face of $G$, there are at most two $3$-faces adjacent to it$\}$. Note that for any plane graph $G$ and any $3$-face $f$ of $G$, there are at most three $3$-faces adjacent to $f$. In our first result, we prove a weakening of Conjecture~\ref{mainconj}(a) by showing that Conjecture~\ref{mainconj}(a) holds for the graphs in the class $\mathcal{G}_1$.

\begin{theorem}\label{two3faceswithone3face}
Let $G\in \mathcal{G}_1$ and $L$ be a list assignment of $G$ such that $|L(v)|\geq 10$ for every $v\in V(G)$. Then diam$(\mathcal{G}(G,L))\leq 190|V(G)|$.
\end{theorem}

Let $\mathcal{G}_2=\{G\mid G$ is a planar graph and $G$ has a plane embedding such that for every $3$-face of $G$, there is at most one $3$-face adjacent to it$\}$. Clearly, $\mathcal{G}_2$ is a proper subclass of $\mathcal{G}_1$. In our second result, we prove that list size $9$ is sufficient for the graphs in the class $\mathcal{G}_2$ to ensure a linear upper bound on the diameter of the recoloring graphs. 

 \begin{theorem}\label{one3facewithone3face}
Let $G\in \mathcal{G}_2$ and $L$ be a list assignment of $G$ such that $|L(v)|\geq 9$ for every $v\in V(G)$. Then diam$(\mathcal{G}(G,L))\leq 13|V(G)|$.
\end{theorem}

 Let $\mathcal{G}_3=\{G\mid $ $G$ is a planar graph and $G$ has a plane embedding such that the faces adjacent to any $3$-face of $G$ are $6^+$-faces$\}$. Clearly, $\mathcal{G}_3$ is a proper subclass of $\mathcal{G}_2$. In our third result, we prove that list size $7$ is sufficient for the graphs in the class $\mathcal{G}_3$ to ensure a linear upper bound on the diameter of the recoloring graphs.

\begin{theorem}\label{trianglewith6+faces}
Let $G\in \mathcal{G}_3$ and $L$ be a list assignment of $G$ such that $|L(v)|\geq 7$ for every $v\in V(G)$. Then diam$(\mathcal{G}(G,L))\leq 242|V(G)|$.
\end{theorem}

 Note that the class $\mathcal{G}_3$ is a superclass of the class of triangle-free planar graphs. So Theorem~\ref{trianglewith6+faces} generalizes the result given by Cranston~\cite{cranstonEJC} that confirms Conjecture~\ref{mainconj}(b). Cranston used the discharging method to obtain a useful structural result for triangle-free planar graphs. Then he used the structural result to obtain a linear upper bound on the recoloring diameter of triangle-free planar graphs. To prove Theorem~\ref{trianglewith6+faces}, we also use the discharging method to obtain a structural result for the graphs in the class $\mathcal{G}_3$ (see Theorem~\ref{strtrianglewith5+faces}). However, our discharging method differs from that of Cranston in the sense that we use the balanced charging method to assign charges to the vertices and the faces of the graph (see Subsection~\ref{balanced}) whereas Cranston used the vertex charging method. It turns out that the collection of unavoidable configurations obtained by Cranston is a proper subset of our collection of unavoidable configurations.
	
For a positive integer $p\geq 3$, we say that a planar graph $G$ is without $p$-cycles if $G$ does not contain any $p$-cycle as a subgraph. The following lemma which is useful to study the recoloring diameter of planar graphs without $p$-cycles follows from a result of Cranston~\cite{cranstonEJC}:

\begin{lemma}[\cite{cranstonEJC}]\label{cranstondiam}
Let $G$ be a $d$-degenerate graph and $L$ be a list assignment of $G$. If $|L(v)|\geq 2d+2$ for every $v\in V(G)$, then diam$(\mathcal{G}(G,L))\leq c\cdot|V(G)|$, where $c$ is an integer such that $\big\lceil \frac{dc}{d+1}\big\rceil +1\leq c.$
\end{lemma}	
It is known that for any $p\in \{3,5,6\}$, the planar graph $G$ without $p$-cycles is $3$-degenerate (see \cite{fijavz,wang}). Therefore, Lemma~\ref{cranstondiam} is applicable if $|L(v)|\geq 8$ for every $v\in V(G)$ and $c$ is fixed as $4$. It follows that diam$(\mathcal{G}(G,L))\leq~4|V(G)|$ for such graphs. A graph $G$ is said to be without mutually adjacent $C_1$-, $C_2$-, and $C_3$-cycles if $G$ does not contain the configuration involving $C_1$, $C_2$, and $C_3$ such that each pair of $C_1$, $C_2$, and $C_3$ are adjacent. Recently, Sittitrai and Nakprasit~\cite{sittitrai} have shown that a planar graph $G$ without mutually adjacent $3$-, $5$-, and $6$-cycles is $3$-degenerate. Therefore, Lemma~\ref{cranstondiam} is applicable if $|L(v)|\geq 8$ for every $v\in V(G)$ and $c$ is fixed as $4$. It follows that diam$(\mathcal{G}(G,L))\leq~4|V(G)|$ for such graphs. Note that the class of planar graphs without mutually adjacent $3$-, $5$-, and $6$-cycles is a superclass of planar graphs without $p$-cycles for some $p\in \{3,5,6\}$. Therefore, this result generalizes the previous results.  To the best of our knowledge, no result is known about the $3$-degeneracy of the planar graphs without $4$-cycles. In our fourth result, we fill this gap by proving that list size $8$ is sufficient for the planar graph without $4$-cycles to ensure a linear upper bound on the diameter of the recoloring graph.
	
\begin{theorem}\label{planarwithout4cycles}
Let $G$ be a planar graph without 4-cycles and $L$ be a list assignment of $G$ such that $|L(v)|\geq 8$ for every $v\in V(G)$. Then diam$(\mathcal{G}(G,L))\leq 29|V(G)|$.
\end{theorem}

For a graph $G$, a set $S\subseteq V(G)$ is an \emph{independent set} of $G$ if the vertices of $S$ are pairwise non-adjacent in $G$. The independence number of a graph $G$ is the maximum cardinality of an independent set of $G$. Recently, Merkel~\cite{merkel} proved that if $G$ is a $k$-colorable graph with the independence number at most~$2$, then diam$(\mathcal{G}(G,k+1))\leq 4|V(G)|$. We generalize the result of Merkel by proving the following theorem.

\begin{theorem}\label{boundedindependence}
Let $G$ be a $k$-colorable graph with the independence number at most $p$. Then diam$(\mathcal{G}(G,\ell))\leq 4|V(G)|$ when $\ell\geq \big\lfloor \frac{p\cdot k}{2}\big\rfloor+1$. Moreover, for $\ell=\big\lfloor \frac{p\cdot k}{2}\big\rfloor$, there exists a $k$-colorable graph $G'$ with the independence number at most $p$ such that $\mathcal{G}(G',\ell)$ is disconnected. 
\end{theorem}

\subsection{Preliminaries}\label{prelim}

 For a plane graph $G$, we use $F(G)$ to denote the set of faces of $G$. For a vertex $v$ of a graph $G$, we use $d_G(v)$ to denote the degree of $v$ in $G$. For a plane graph $G$ and a face $f\in F(G)$, we use $d_G(f)$ to denote the length of it, that is the number of edges incident on its boundary. Note that any cut edge incident to $f$ contributes $2$ to $d_G(f)$. The \emph{open neighborhood} of a vertex $v$, denoted by $N_G(v)$, is the set of vertices adjacent to $v$. We use $d(v)$, $d(f)$, and $N(v)$ to denote $d_G(v)$, $d_G(f)$, and $N_G(v)$, respectively if the context of the graph is clear. Let $k$ be a non-negative integer. A \emph{$k$-vertex} is a vertex of degree $k$. A \emph{$k^+$-vertex} (respectively \emph{$k^-$-vertex}) is a vertex of degree at least (respectively at most) $k$. A \emph{$k/k^+/k^-$-neighbor} of a vertex $v$ is a neighbor of $v$ that is a $k/k^+/k^-$-vertex. Let $G$ be a plane graph and $f\in F(G)$. The face $f$ is a \emph{$k$-face} if $d(f)=k$. The face $f$ is a \emph{$k^+$-face} (resp. \emph{$k^-$-face}) if the degree of $f$ is at least (resp. at most) $k$. Let $G$ be a graph and $S\subseteq V(G)$. We use $G[S]$ to denote the subgraph of $G$ induced by the vertices of $S$ and $G-S$ to denote the subgraph $G[V(G)\setminus S]$. In particular, if $S=\{v\}$, then we simply write $G-v$ to denote $G-S$. Given a path $P=v_1v_2\ldots v_p$ and two vertices $v,w\notin V(P)$ such that $v$ is adjacent to $v_1$ and $w$ is adjacent to $v_p$, we use $vP$ and $Pw$ to denote the paths $vv_1v_2\ldots v_p$ and $v_1v_2\ldots v_pw$, respectively. 

Let $G$ be a $L$-colorable graph, where $L$ is a list assignment of $G$. Let $\alpha$ and $\beta$ be two $L$-colorings of $G$. A \emph{recoloring sequence} $\sigma_G$ from $\alpha$ to $\beta$ is a path in $\mathcal{G}(G,L)$ between the vertices corresponding to $\alpha$ and $\beta$; clearly, this path corresponds to a sequence of $L$-colorings of $G$ such that two consecutive $L$-colorings differ at exactly one vertex of $G$. For a subgraph $H$ of $G$, we use $\alpha|_H$ to denote the restriction of $\alpha$ to $H$. Let $\sigma_H$ be a recoloring sequence of $H$ from $\alpha|_H$ to $\beta|_H$. Now $\sigma_G=\alpha^0\alpha^1\ldots\alpha^t$, with $\alpha^0=\alpha$ and $\alpha^t=\beta$, is an extension of $\sigma_H$ to $G$ if $\sigma_H$ is obtained from $\alpha|^0_H \alpha|^1_H\ldots \alpha|^t_H$ by deleting all but exactly one term from any subsequence of consecutive terms that repeat themselves. Note that if a vertex of $H$ is recolored $k$ times in $\sigma_H$, then it is recolored exactly $k$ times in any extension $\sigma_G$ of it. We refer to \cite{diestel} for the graph theoretic notations that are not defined here.  

\begin{lemma}[\cite{bartierarxiv,cranstonEJC}]\label{keylemma}
Let $L$ be a list assignment of a graph $G^*$, $\alpha$ and $\beta$ be two $L$-colorings of $G^*$, and $G^{**}=G^*-v$, where $v$ is a vertex of $G^*$ such that $|L(v)|\geq d_{G^*}(v)+2$. Let $\sigma_{G^{**}}$ be a recoloring sequence from $\alpha|_{G^{**}}$ to $\beta|_{G^{**}}$ and $t$ be the total number of recolorings of the vertices of $N_{G^*}(v)$ in $\sigma_{G^{**}}$. Then $\sigma_{G^{**}}$ can be extended to a recoloring sequence $\sigma_{G^*}$ from $\alpha$ to $\beta$ that recolors $v$ at most $\big\lceil \frac{t}{|L(v)|-d_{G^*}(v)-1}\big\rceil+1$ times.
\end{lemma}

\subsection{Balanced charging}\label{balanced}

To prove Theorem~\ref{two3faceswithone3face}-\ref{trianglewith6+faces}, we use the discharging method. By the discharging method, we obtain some unavoidable configurations in the graphs of the considered classes. These unavoidable configurations allow us to use Lemma~\ref{keylemma} to obtain a recoloring sequence of size linear in the number of vertices.

Balanced charging is a well known method to assign charges to the vertices and the faces of a plane graph $G$. Each vertex $v\in V(G)$ is assigned the charge $d(v)-4$, and each face $f\in F(G)$ is assigned the charge $d(f)-4$. Let $ch(G)$ denote the total sum of the charges assigned to the vertices and the faces of $G$. That is,
$$ch(G)= \sum_{v\in V(G)}(d(v)-4)+\sum_{f\in F(G)}(d(f)-4).$$
Since every edge of $G$ contributes 2 to the vertex degree sum and 2 to the face length sum, we have $\displaystyle\sum_{v\in V(G)}d(v)=2|E(G)|$ and $\displaystyle\sum_{f\in F(G)}d(f)=2|E(G)|$. So we have
$$ch(G)= (2|E(G)|-4|V(G)|)+(2|E(G)|-4|F(G)|)=-4(|V(G)|+|F(G)|-|E(G)|).$$
Hence by Euler's formula, we have $ch(G)=-8$.

\section{Proof of Theorem~\ref{two3faceswithone3face}}

In a plane graph, we use \emph{diamond} to denote the configuration of two adjacent $3$-faces with only one edge common. Given a diamond, we use \emph{mid-edge} to denote the common edge of the adjacent $3$-faces. A diamond is said to be incident to a vertex $v$ if the mid-edge of the diamond is incident to $v$. Let $G$ be a plane graph and $v\in V(G)$. A diamond incident to $v$ is said to be a \emph{special diamond} if an edge of the diamond that is incident to $v$ (other than the mid-edge) is also incident to a $4^+$-face. In Figure~\ref{special}(a), the diamond induced by $\{v,v_1,v_2,v_3\}$ is a special diamond incident to $v$ whereas the diamond induced by $\{v,v_2,v_3,v_4\}$ is not a special diamond incident to $v$. Let $w_G^t(v)$ and $w_G^d(v)$ denote the total number of $3$-faces and the total number of special diamonds incident to $v$, respectively and let $w_G(v)=w_G^t(v)+w_G^d(v)$.

\begin{figure}[hbtp]
\centering
\includegraphics[scale=.8]{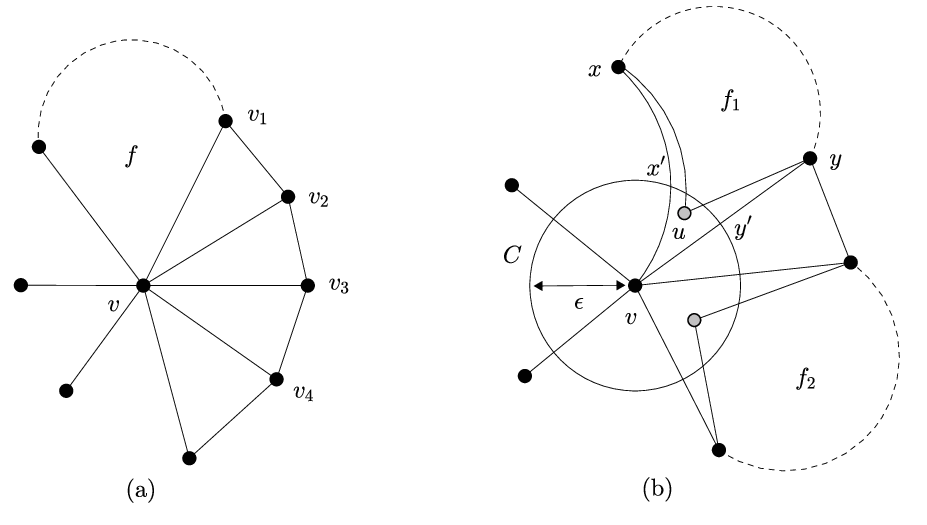}
\caption{The faces $f,f_1,$ and $f_2$ are $4^+$-faces.}\label{special}
\end{figure}

\begin{obs}\label{obser}
Given a plane graph $G$ and a fixed vertex $v$ of $G$, there exists a plane graph $G'$ such that \begin{enumerate*}
\item $d_{G'}(v)=d_{G}(v)$, \item there are $d_{G'}(v)$ distinct faces incident to $v$ in $G'$, and \item $w_{G'}(v)=w_{G}(v)$.
\end{enumerate*}
\end{obs}
\begin{proof}
 Assume that there are less than $d_G(v)$ distinct faces incident to $v$. Then we move around $v$ in the clockwise manner and visit the edges incident to $v$. Consider a circle $C$ with the centre $v$ and sufficiently small radius, say $\epsilon$. Assume that there are two consecutive edges $vx$ and $vy$ incident to a $4^+$-face $f_1$. Then we add a vertex $u$ in the inner region of the sector $x'vy'$, where $x'$ and $y'$ are the points on the intersection of $C$ with the edges $vx$ and $vy$, respectively. Then add edges $ux$ and $uy$ such that $ux$ and $uy$ are drawn closely to $vx$ and $vy$, respectively (see Figure~\ref{special}(b)). Let $G'$ be the resulting graph. Note that the motivation of such a construction of $G'$ is to ensure that $G'$ is a plane graph. Further note that $d_{G'}(v)=d_{G}(v)$ and there are exactly $d_{G'}(v)$ distinct faces incident to $v$. Moreover, since the total number of $3$-faces and special diamonds incident to $v$ remains the same, we have $w_{G'}(v)=w_{G}(v)$.
\end{proof}

\begin{lemma}\label{totaltrianglemidedge}
Let $G$ be a plane graph and $v\in V(G)$ such that $d_G(v)\geq 7$. Then $w_G(v)\leq 3(d_G(v)-4)$.
\end{lemma}
\begin{proof}
 Without loss of generality, by Observation~\ref{obser}, we may assume that there are $d_G(v)$ distinct faces incident to $v$. We use induction on $d_G(v)$ to prove that for any plane graph $G$ and $v\in V(G)$ with $d_G(v)\geq 7$, we have $d_G(v)\leq 3(d_G(v)-4)$. 

To show the base case, assume that $G$ is a plane graph and $v\in V(G)$ with $d_G(v)=7$. We need to show that $w_G(v)\leq 9$. If all the faces incident to $v$ are $3$-faces, then there is no special diamond incident to $v$ and hence $w_G(v)=w_G^t(v)=7< 9$. Assume that there is exactly one $4^+$-face incident to $v$. Then there are six $3$-faces incident to $v$ and hence $w_G^t(v)=6$. Note that the $4^+$-face causes two diamonds incident to $v$ to be the special diamonds. So $w_G^d(v)=2$ and hence $w_G(v)=2+6=8$. Now assume that there are at least two $4^+$-faces incident to $v$. Then there are at most five $3$-faces incident to $v$ and hence $w_G^t(v)\leq 5$. Note that an edge incident to both $v$ and a $4^+$-face cannot be the mid-edge of a special diamond incident to $v$. Since there are at least two $4^+$-faces incident to $v$, at least three edges incident to $v$ are not mid-edges of special diamonds incident to $v$. So at most four edges incident to $v$ are the mid-edges of special diamonds incident to $v$ and hence there are at most four special diamonds incident to $v$. So we have $w_G^d(v)\leq 4$. Hence $w_G(v)=w_G^t(v)+w_G^d(v)\leq 5+4=9$. Therefore, the base case of the induction is true.
 
Now assume that the induction hypothesis is true when $d_G(v)<k$, that is for any plane graph $G$ and any vertex $v\in V(G)$ such that $7\leq d(v)<k$, we have $w_G(v)\leq 3(d_G(v)-4)$. Let $d_G(v)=k$. We need to show that $w_G(v)\leq 3(k-4)$. If there is no special diamond incident to $v$, then $w_G^d(v)=0$ and $w_G^t(v)\leq d_G(v)=k$ and hence $w_G(v)\leq k\leq 3(k-4)$. So assume that there is a special diamond  induced by $\{v,v_1,v_2,v_3\}$ with $vv_2$ as the mid-edge. Now consider a plane graph $G'$ obtained by removing the edge $vv_2$ and adding the edge $v_1v_3$ in $G$. Since $d_{G'}(v)=k-1$, by the induction hypothesis, $w_{G'}(v)\leq 3(k-1-4)=3(k-5)$. Note that if $vv_1$ is the mid-edge of a special diamond incident to $v$ in $G$, then $vv_1$ is also the mid-edge of a special diamond incident to $v$ in $G'$. Similarly, if $vv_3$ is the mid-edge of a special diamond incident to $v$ in $G$, then $vv_3$ is also the mid-edge of a special diamond incident to $v$ in $G'$. So $v$ has at most one more mid-edge of a special diamond, that is possibly $vv_2$, and one more $3$-face incident to $v$ in $G$ than in $G'$. So $w_G(v)-w_{G'}(v)\leq 2$. Hence $w_G(v)\leq w_{G'}(v)+2\leq 3(k-5)+2< 3(k-4)$. This completes the proof of the lemma.
\end{proof}

\begin{figure}[hbtp]
\centering
\includegraphics[scale=.9]{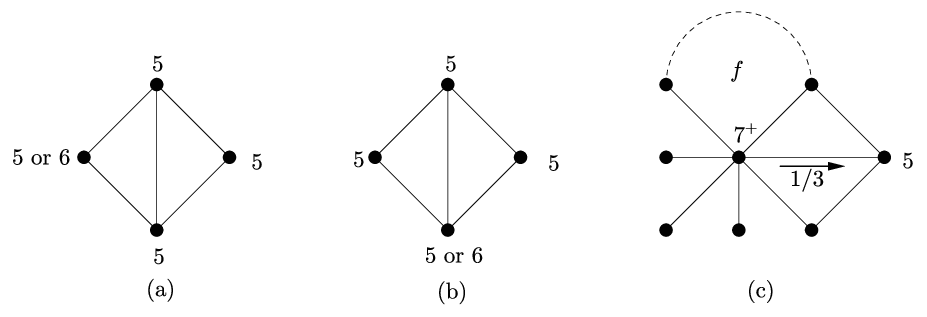}
\caption{The numbers next to the vertices represent their degrees and $f$ is a $4^+$-face. In (c), we partially represent the discharging Rule~$R_3$ given in the proof of Theorem~\ref{strtwo3faceswithone3face}.}\label{figure2}
\end{figure}

\begin{theorem}\label{strtwo3faceswithone3face}
If $G\in \mathcal{G}_1$ and $\delta(G)\geq 5$, then $G$ contains a diamond with three vertices of degree $5$ and one vertex of degree at most $6$ (see Figure~\ref{figure2}(a)-(b)).
\end{theorem}
\begin{proof}
Let $G\in \mathcal{G}_1$ and $\delta(G)\geq 5$. Consider a plane embedding of $G$ in which for each $3$-face of $G$, there are at most two $3$-faces adjacent to it. For the sake of contradiction, assume that $G$ does not contain the configuration $H$, where $H$ is a diamond with three vertices of degree $5$ and one vertex of degree at most $6$ (see Figure~\ref{figure2}(a)-(b)). Note that two $3$-faces cannot share more than one edge since $\delta(G)\geq 5$. Assign charges to the vertices and the faces of $G$ using the balanced charging method (see Subsection~\ref{balanced}). Now we re-allocate the charges using the following discharging rules.
\begin{enumerate}
\item[$R_1$:] Every vertex sends charge $1/3$ to every $3$-face incident to it.
\item[$R_2$:] Let $v$ be a $6$-vertex. If there is a $5$-vertex $u$ such that $uv$ is the mid-edge of a special diamond incident to $v$, then $v$ sends charge $1/6$ to $u$.
\item[$R_3$:] Let $v$ be a $7^+$-vertex. If there is a $5$-vertex $u$ such that $uv$ is the mid-edge of a special diamond incident to $v$, then $v$ sends charge $1/3$ to $u$. In Figure~\ref{figure2}(c), Rule~$R_3$ is illustrated.
\end{enumerate}
Now we show that at the end of the discharging process, every vertex and every face has non-negative charge to obtain a contradiction to the fact that $ch(G)=-8$.\medskip

\textbf{$3$-face:} Let $f$ be a $3$-face. Initially, $f$ had charge $-1$. Since $f$ gets charge $1/3$ from each of its incident vertices by Rule~$R_1$, $f$ has non-negative charge at the end of the discharging process.

\textbf{$4^+$-face:} Let $f$ be a $4^+$-face. Initially, $f$ had non-negative charge. Since $f$ neither gains nor loses any charge during the discharging process, its charge remains the same, that is non-negative. 

\textbf{$5$-vertex:} Let $v$ be a $5$-vertex of $G$. Initially, $v$ had charge $1$. If there are at most three $3$-faces incident to $v$, then $v$ loses charge at most $3\cdot(1/3)=1$ and hence has non-negative charge at the end of the discharging process. So we may assume that there are at least four $3$-faces incident to $v$. 

\noindent\textbf{Case~1.} There are five $3$-faces incident to $v$. 

Note that $v$ loses charge $5\cdot (1/3)$ to the faces incident to it. Since $v$ can afford to lose charge at most $1$, it must regain charge $2/3$. Note that for every $3$-face incident to $v$, there are two $3$-faces adjacent to it. Hence by the definition of the class $\mathcal{G}_1$, for each of the $3$-faces incident to $v$, there must be a $4^+$-face adjacent to it. We can infer that if $u$ is a neighbor of $v$, then the edge $uv$ is the mid-edge of a special diamond incident to $u$. If $v$ has at least two $7^+$-neighbors, then by Rule~$R_3$, $v$ gets sufficient charge. Assume that $v$ has exactly one $7^+$-neighbor. Then $v$ has four remaining $6^-$-neighbors. Since $\delta(G)\geq 5$, each of these remaining four neighbors of $v$ is a $5$-vertex or a $6$-vertex. If at most one of these four neighbors of $v$ is $6$-vertex, then $G[N(v)\cup \{v\}]$ contains the configuration $H$, a contradiction. So at least two among these four $6^-$-neighbors of $v$ are $6$-vertices.  
 Hence $v$ gets charge $1/3$ from the $7^+$-neighbor by Rule~$R_3$ and at least $2\cdot (1/6)$ from the $6$-neighbors by Rule~$R_2$. Hence $v$ gets sufficient charge to make its charge non-negative. Finally, assume that $v$ has no $7^+$-neighbor. If $v$ has at most three $6$-neighbors and the remaining are $5$-neighbors, then it is easy to verify that $G[N(v)\cup \{v\}]$ contains the configuration $H$, a contradiction. So $v$ has at least four $6$-neighbors. Hence $v$ gets sufficient charge, that is at least $4\cdot (1/6)$, by Rule~$R_2$ to make its charge non-negative. Hence $v$ has non-negative charge at the end of the discharging process.

\noindent\textbf{Case~2.} There are exactly four $3$-faces incident to $v$.

Note that $v$ loses charge $4\cdot (1/3)$ to the faces incident to it by Rule~$R_1$. Since $v$ can afford to lose charge at most $1$, it must regain charge $1/3$. Let $v_1,v_2,v_3,v_4$, and $v_5$ be the vertices adjacent to $v$ such that $\{v,v_i,v_{i+1}\}$ induces a $3$-face for every $i\in \{1,2,3,4\}$. Since $G\in \mathcal{G}_1$, for each $3$-face of $G$, there are at most two $3$-faces adjacent to it. So for each of the $3$-faces induced by $\{v,v_2,v_3\}$ and $\{v,v_3,v_4\}$, there must be a $4^+$-face adjacent to it. Note that $vv_2$, $vv_3$, and $vv_4$ are the mid-edges of special diamonds incident to $v_2$, $v_3$, and $v_4$, respectively. Now if any of $v_2,v_3,$ and $v_4$ is a $7^+$-vertex, then $v$ gets sufficient charge by Rule~$R_3$. So assume that each of $v_2,v_3,$ and $v_4$ is a $5$-vertex or a $6$-vertex. Now if at most one among $v_2,v_3,$ and $v_4$ is a $6$-vertex, then $G[\{v,v_2,v_3,v_4\}]$ contains the configuration $H$, a contradiction. So at least two of $v_2,v_3,$ and $v_4$ are $6$-vertices. Hence $v$ gets sufficient charge, that is $2\cdot (1/6)$, by Rule~$R_2$ and hence $v$ has non-negative charge at the end of the discharging process.

\textbf{$6$-vertex:} Let $v$ be a $6$-vertex. Initially, $v$ had charge $2$. Let the total number of $4^+$-faces incident to $v$ be $t$. Then the number of $3$-faces incident to $v$ is at most $6-t$. So $v$ loses charge at most $\frac{1}{3}(6-t)$ to the $3$-faces incident to it by Rule~$R_1$.  We claim that there are at most $2t$ special diamonds incident to $v$. If each $4^+$-face incident to $v$ contains only two incident edges to $v$ on its boundary, then this claim is trivial since each $4^+$-face can cause at most two diamonds incident to $v$ to be special diamonds. If some of the $4^+$-faces have exactly three edges incident to $v$ on its boundary, then it is easy to see that the situation does not change. So assume that there is a $4^+$-face having four or more edges incident to $v$ on its boundary. Then these edges are not the mid-edges of any special diamonds incident to $v$. It follows that at most $2$ special diamonds can be there which is clearly at most $2t$. So the claim holds. Therefore, $v$ loses charge at most $\frac{1}{6}(2t)$ through the mid-edges of the special diamonds incident to it by Rule~$R_2$. So $v$ loses charge at most $\frac{1}{3}(6-t)+\frac{1}{6}(2t)=2$ and hence $v$ has non-negative charge at the end of the discharging process.

\textbf{$7^+$-vertex:} Let $v$ be a $7^+$-vertex. Initially, $v$ had charge $d_G(v)-4$. Recall that $w_G(v)=w_G^t(v)+w_G^d(v)$, where $w_G^t(v)$ is the total number of the $3$-faces incident to $v$ and $w_G^d(v)$ is the total number of the special diamonds incident to $v$.  Since $v$ loses charge $1/3$ to each of the incident $3$-faces by Rule~$R_1$, $v$ loses charge $\frac{1}{3}w_G^t(v)$ to the incident $3$-faces. Recall that   by Rule~$R_3$, $v$ loses charge $1/3$ to a $5$-vertex $u$ if $uv$ is the mid-edge of a special diamond incident to $v$. Note that there are at most $w_G^d(v)$ such $5$-vertices. So $v$ loses charge at most $\frac{1}{3}w_G^d(v)$ due to Rule~$R_3$. Hence $v$ loses charge at most $\frac{1}{3}w_G^t(v)+\frac{1}{3}w_G^d(v)=\frac{1}{3}w_G(v)$. Since $d_G(v)\geq 7$, by Lemma~\ref{totaltrianglemidedge}, $w_G(v)\leq 3(d_G(v)-4)$. So $v$ loses charge at most $\frac{1}{3}\cdot 3(d_G(v)-4)=d_G(v)-4$ and hence $v$ has non-negative charge at the end of the discharging process.
\end{proof}


\begin{proof}[Proof of Theorem~\ref{two3faceswithone3face}]

Let $G\in \mathcal{G}_1$ and $\alpha$ and $\beta$ be any two $L$-colorings of $G$. In order to prove that diam$(\mathcal{G}(G,L))\leq 190|V(G)|$, we prove the claim that there exists a recoloring sequence from $\alpha$ to $\beta$ that recolors every vertex of $G$ at most $190$ times. We use induction on the number of vertices of $G$ to prove the claim. If $|V(G)|=1$, then the claim is trivial. So the base case is true. Now assume that there is a $4^-$-vertex $v\in V(G)$. By induction, there is a recoloring sequence $\sigma_{G-v}$ from $\alpha|_{G-v}$ to $\beta|_{G-v}$ that recolors every vertex of $G-v$ at most $190$ times. Note that the total number of recolorings of the vertices of $N(v)$ in $\sigma_{G-v}$ is at most $t=d(v)\cdot 190$. So by taking $G^*=G$, $G^{**}=G-v$, and $L(v)=10$ in Lemma~\ref{keylemma}, the sequence $\sigma_{G-v}$ can be extended to a recoloring sequence $\sigma_G$ from $\alpha$ to $\beta$ that recolors $v$ at most $\big\lceil \frac{d(v)\cdot 190}{10-d(v)-1}\big\rceil +1\leq 153<190$ times since $d(v)\leq 4$. 
Now we may assume that $\delta(G)\geq 5$. Then by Theorem~\ref{strtwo3faceswithone3face}, $G$ contains a diamond $H$ with three vertices of degree~$5$ and one vertex of degree at most~$6$. Let $V(H)=\{v_1,v_2,v_3,v_4\}$, where $v_1$ is the $6^-$-vertex and  $v_2,v_3$, and $v_4$ are $5$-vertices. For the rest of the proof, we assume that $v_1$ is a $6$-vertex. For the case when $v_1$ is a $5$-vertex, the proof is similar.

First assume that $v_1v_3$ is the mid-edge of $H$. By induction, there exists a recoloring sequence $\sigma_{G-V(H)}$ from $\alpha|_{G-V(H)}$ to $\beta|_{G-V(H)}$ that recolors every vertex of $G-V(H)$ at most $190$ times. Note that $d_{G-\{v_2,v_3,v_4\}}(v_1)=3$ and the total number of recolorings of the vertices of $N(v_1)\cap (G-V(H))$ in $\sigma_{G-V(H)}$ is at most $t=3\cdot 190$. So by taking $G^*=G-\{v_2,v_3,v_4\}$, $G^{**}=G-V(H)$, and $L(v)=10$ in Lemma~\ref{keylemma}, $\sigma_{G-V(H)}$ can be extended to a recoloring sequence $\sigma_{G-\{v_2,v_3,v_4\}}$ from $\alpha|_{G-\{v_2,v_3,v_4\}}$ to $\beta|_{G-\{v_2,v_3,v_4\}}$ that recolors $v_1$ at most $\big\lceil \frac{3\cdot 190}{10-3-1}\big\rceil +1=96<190$ times. Note that $d_{G-\{v_2,v_4\}}(v_3)=3$ and the total number of recolorings of the vertices of $N(v_3)\cap  (G-\{v_2,v_3,v_4\})$ in $\sigma_{G-\{v_2,v_3,v_4\}}$ is at most $t=2\cdot 190+96$. So by taking $G^*=G-\{v_2,v_4\}$, $G^{**}=G-\{v_2,v_3,v_4\}$, and $L(v)=10$ in Lemma~\ref{keylemma}, $\sigma_{G-\{v_2,v_3,v_4\}}$ can be extended to a recoloring sequence $\sigma_{G-\{v_2,v_4\}}$ from $\alpha|_{G-\{v_2,v_4\}}$ to $\beta|_{G-\{v_2,v_4\}}$ that recolors $v_3$ at most $\big\lceil \frac{2\cdot 190+96}{10-3-1}\big\rceil +1=81<190$ times. Similarly, $\sigma_{G-\{v_2,v_4\}}$ can be extended to a recoloring sequence $\sigma_{G-v_4}$ from $\alpha|_{G-v_4}$ to $\beta|_{G-v_4}$ and then $\sigma_{G-v_4}$ can be extended to a recoloring sequence $\sigma_{G}$ from $\alpha$ to $\beta$  that recolors each of $v_2$ and $v_4$ at most $\big\lceil \frac{3\cdot 190+96+81}{10-5-1}\big\rceil +1=188<190$ times.

Now assume that $v_2v_4$ is the mid-edge of $H$. By induction, there exists a recoloring sequence $\sigma_{G-V(H)}$ from $\alpha|_{G-V(H)}$ to $\beta|_{G-V(H)}$ that recolors every vertex of $G-V(H)$ at most $190$ times. Note that $d_{G-\{v_2,v_3,v_4\}}(v_1)\leq 4$ and the total number of recolorings of the vertices of $N(v_1)\cap (G-V(H))$ in $\sigma_{G-V(H)}$ is at most $t=4\cdot 190$. So by taking $G^*=G-\{v_2,v_3,v_4\}$, $G^{**}=G-V(H)$, and $L(v)=10$ in Lemma~\ref{keylemma}, $\sigma_{G-V(H)}$ can be extended to a recoloring sequence $\sigma_{G-\{v_2,v_3,v_4\}}$ from $\alpha|_{G-\{v_2,v_3,v_4\}}$ to $\beta|_{G-\{v_2,v_3,v_4\}}$ that recolors $v_1$ at most $\big\lceil \frac{4\cdot 190}{10-4-1}\big\rceil +1=153<190$ times. Note that $d_{G-\{v_2,v_4\}}(v_3)= 3$ and the total number of recolorings of the vertices of $N(v_3)\cap (G-\{v_2,v_3,v_4\})$ in $\sigma_{G-\{v_2,v_3,v_4\}}$ is at most $t=3\cdot 190$. Now by taking $G^*=G-\{v_2,v_4\}$, $G^{**}=G-\{v_2,v_3,v_4\}$, and $L(v)=10$ in Lemma~\ref{keylemma}, $\sigma_{G-\{v_2,v_3,v_4\}}$ can be extended to a recoloring sequence $\sigma_{G-\{v_2,v_4\}}$ from $\alpha|_{G-\{v_2,v_4\}}$ to $\beta|_{G-\{v_2,v_4\}}$ that recolors $v_3$ at most $\big\lceil \frac{3\cdot 190}{10-3-1}\big\rceil +1=96<190$ times. Similarly, $\sigma_{G-\{v_2,v_4\}}$ can be extended to a recoloring sequence $\sigma_{G-v_4}$ from $\alpha|_{G-v_4}$ to $\beta|_{G-v_4}$ that recolors $v_2$ at most $\big\lceil \frac{2\cdot 190+153+96}{10-4-1}\big\rceil +1=127<190$ times. Then $\sigma_{G-v_4}$ can be extended to a recoloring sequence $\sigma_{G}$ from $\alpha$ to $\beta$  that recolors $v_4$ at most $\big\lceil \frac{2\cdot 190+153+127+96}{10-5-1}\big\rceil +1=190$ times.
\end{proof}

\section{Proof of Theorem~\ref{one3facewithone3face}}

%

To prove Theorem~\ref{one3facewithone3face}, we first prove the following structural result for the graphs in the class $\mathcal{G}_2$ by using the discharging method.

\begin{theorem}\label{strone3facewithone3face}
If $G\in \mathcal{G}_2$ and $\delta(G)\geq 4$, then $G$ contains at least one of the following configurations (see Figure~\ref{figure3}).

\begin{figure}[hbtp]
\centering
\includegraphics[scale=0.75]{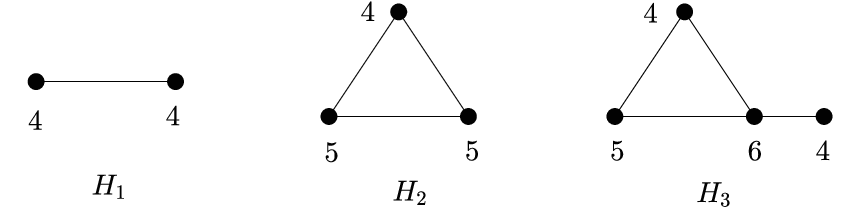}
\caption{The numbers next to the vertices represent their degrees.}\label{figure3}
\end{figure}

\begin{enumerate}
\item[$H_1$:] An edge whose both endpoints have degree $4$.
\item[$H_2$:] A triangle with one vertex of degree $4$ and the other two vertices of degree $5$.
\item[$H_3$:] A triangle with the three vertices having degrees $4$, $5$, and $6$ such that the $6$-vertex of the triangle has a $4$-neighbor that is not a vertex of the triangle. 
\end{enumerate}
\end{theorem}
\begin{proof}
Let $G\in\mathcal{G}_2$ and $\delta(G)\geq 4$. Consider a plane embedding of $G$ such that for each $3$-face of $G$, there is at most one $3$-face adjacent to it. For the sake of contradiction, assume that $G$ does not contain any of the configurations $H_1, H_2$, and $H_3$. Note that two $3$-faces of $G$ cannot share more than one edge since $\delta(G)\geq 4$. Assign charges to the vertices and the faces of $G$ by using the balanced charging method (see Subsection~\ref{balanced}). Now we re-allocate the charges using the following discharging rules.
\begin{enumerate}
\item[$R_1$:] Every $5$-vertex sends charge $1/3$ to every $3$-face incident to it.
\item[$R_2$:] Let $v$ be a $6^+$-vertex and $f$ be a $3$-face incident to it. If the other two vertices incident to $f$ are a $4$-vertex and a $5$-vertex, then $v$ sends charge $2/3$ to $f$. If the other two vertices incident to $f$ are a $4$-vertex and a $6^+$-vertex, then $v$ sends charge $1/2$ to $f$. If both the other vertices incident to $f$ are $5^+$-vertices, then $v$ sends charge $1/3$ to $f$.  
\end{enumerate}

Now we show that at the end of the discharging process, every vertex and every face has non-negative charge to obtain a contradiction to the fact that $ch(G)=-8$.\medskip

\textbf{$3$-face:} Let $f$ be a $3$-face of $G$. Initially, $f$ had charge $-1$. Since $\delta(G)\geq 4$, all the vertices incident to $f$ are $4^+$-vertices. First assume that all the vertices incident to $f$ are $5^+$-vertices. Then $f$ gets charge $1/3$ from each of its incident vertices by Rules~$R_1$ and $R_2$. Hence $f$ has non-negative charge at the end of the discharging process. Now assume that there is a $4$-vertex $v$ incident to $f$. Let $u$ and $w$ be the other two vertices incident to $f$. Since $G$ does not contain the configuration $H_1$, both $u$ and $w$ are $5^+$-vertices. If both $u$ and $w$ are $6^+$-vertices, then $f$ gets charge $1/2$ from each of $u$ and $w$ by Rule~$R_2$ and hence gets sufficient charge to make its charge non-negative. So assume that at least one of $u$ and $w$ is a $5$-vertex. Without loss of generality, we may assume that $u$ is a $5$-vertex. Now since $G$ does not contain the configuration $H_2$, $w$ is a $6^+$-vertex. So $f$ gets charge $1/3$ from $u$ by Rule~$R_1$ and charge $2/3$ from $w$ by Rule~$R_2$. Hence $f$ has non-negative charge at the end of the discharging process.

\textbf{$4^+$-face:} Let $f$ be a $4^+$ face. Initially, $f$ had non-negative charge. Since $f$ neither gains nor loses any charge during the discharging process, its charge remains the same, that is non-negative.

\textbf{$4$-vertex:} Let $v$ be a $4$-vertex. Initially, $v$ had charge~$0$. Since $v$ neither gains nor loses any charge during the discharging process, its charge remains the same, that is~$0$.

\textbf{$5$-vertex:} Let $v$ be a $5$-vertex. Initially, $v$ had charge $1$. If there are at least four $3$-faces incident to $v$, then there is a $3$-faces incident to $v$ which is adjacent to at least two distinct $3$-faces, a contradiction to the fact that $G\in \mathcal{G}_2$. So there are at most three $3$-faces incident to $v$. Now since $v$ sends charge $1/3$ to each of the incident $3$-faces by Rule~$R_1$, $v$ loses charge at most $3\cdot(1/3)$. Hence $v$ has non-negative charge at the end of the discharging process.

\textbf{$6$-vertex:} Let $v$ be a $6$-vertex. Initially, $v$ had charge $2$. If there are at least five $3$-faces incident to $v$, then there is a $3$-face incident to $v$ which is adjacent to at least two distinct $3$-faces, a contradiction to the fact that $G\in \mathcal{G}_2$. So there are at most four $3$-faces incident to $v$. Since $\delta(G)\geq 4$, the neighbors of $v$ are $4^+$-vertices. If every vertex of every $3$-face incident to $v$ is a $5^+$-vertex, then $v$ sends charge $1/3$ to each of the $3$-faces incident to it by Rule~$R_2$. Thus $v$ loses charge at most $4\cdot(1/3)< 2$. So assume that there is a $4$-vertex incident to a $3$-face that is incident to $v$. Since $\delta(G)\geq 4$ and $G$ does not contain the configuration $H_1$, the neighbors of any $4$-vertex are $5^+$-vertices.

First assume that in every $3$-face incident to $v$ that has an incident $4$-vertex, the other incident vertex is a $6^+$-vertex. Since $v$ loses charge at most $1/2$ to each such $3$-face by Rule~$R_2$, $v$ loses charge at most $4\cdot(1/2)=2$. Next assume that there is a $3$-face $f$ incident to $v$ in which the other two incident vertices are a $4$-vertex, say $u$ and a $5$-vertex, say $w$. Then, since $G$ contains the $3$-face $f$ and does not contain the configuration $H_3$, the neighbors of $v$ other than $u$ are $5^+$-vertices. So $v$ loses charge $1/3$ to each of the incident $3$-faces which are not incident to $u$ by Rule~$R_2$. Note that there are at most two $3$-faces that are incident to the edge $uv$ and may get charge $2/3$ from $v$ by Rule~$R_2$. So $v$ loses charge at most $2\cdot\frac{2}{3}+2\cdot\frac{1}{3}=2$. Hence at the end of the discharging process, $v$ has non-negative charge.

\textbf{$7^+$-vertex:} Let $v$ be a $7^+$-vertex. Initially, $v$ had charge $d(v)-4$. Since for each $3$-face of $G$, there are at most two $3$-faces adjacent to it, the number of $3$-faces incident to $v$ is at most $\big\lfloor \frac{2}{3}\cdot d(v)\big\rfloor$. Since $v$ loses charge at most $2/3$ to each incident $3$-face by Rule~$R_2$, it loses charge at most $\frac{2}{3}\cdot\big\lfloor \frac{2}{3}\cdot d(v)\big\rfloor$. Since $d(v)\geq 7$, we have $\frac{2}{3}\cdot\big\lfloor \frac{2}{3}\cdot d(v)\big\rfloor\leq d(v)-4$. Hence at the end of the discharging process, $v$ has non-negative charge.
\end{proof}

\begin{proof}[Proof of Theorem~\ref{one3facewithone3face}]
Let $G\in \mathcal{G}_2$ and $\alpha$ and $\beta$ be any two $L$-colorings of $G$. In order to prove that diam$(\mathcal{G}(G,L))\leq 13|V(G)|$, we prove the claim that there exists a recoloring sequence from $\alpha$ to $\beta$ that recolors every vertex of $G$ at most $13$ times. We use induction on the number of vertices of $G$ to prove the claim. If $|V(G)|=1$, then the claim is trivial. So the base case is true. First assume that there is a $3^-$-vertex $v\in V(G)$. By induction, there is a recoloring sequence $\sigma_{G-v}$ from $\alpha|_{G-v}$ to $\beta|_{G-v}$ that recolors every vertex of $G-v$ at most $13$ times. Note that the total number of recolorings of the vertices of $N(v)$ in $\sigma_{G-v}$ is at most $t=d(v)\cdot 13$. So by taking $G^*=G$, $G^{**}=G-v$, and $|L(v)|=9$ in Lemma~\ref{keylemma}, the sequence $\sigma_{G-v}$ can be extended to a recoloring sequence $\sigma_G$ from $\alpha$ to $\beta$ that recolors $v$ at most $\big\lceil \frac{d(v)\cdot 13}{9-d(v)-1}\big\rceil +1\leq 9<13$ times since $d(v)\leq 3$. 
 Now we may assume that $\delta(G)\geq 4$. Then by Theorem~\ref{strone3facewithone3face}, $G$ contains at least one of the configurations $H_1, H_2$, and $H_3$. Now for each configuration, we show that the claim holds.

\textbf{Configuration $H_1$:} Let $uv$ be an edge of $G$ such that $u$ and $v$ are $4$-vertices. By induction, there exists a recoloring sequence $\sigma_{G-\{u,v\}}$ from $\alpha|_{G-\{u,v\}}$ to $\beta|_{G-\{u,v\}}$ that recolors every vertex of $G-\{u,v\}$ at most $13$ times. Note that $d_{G-u}(v)=3$ and the total number of recolorings of the vertices of $N(v)\cap (G-\{u,v\})$ in $\sigma_{G-\{u,v\}}$ is at most $t=3\cdot 13$. So by taking $G^*=G-u$, $G^{**}=G-\{u,v\}$, and $|L(v)|=9$ in Lemma~\ref{keylemma}, $\sigma_{G-\{u,v\}}$ can be extended to a recoloring sequence $\sigma_{G-u}$ from $\alpha|_{G-u}$ to $\beta|_{G-u}$ that recolors $v$ at most $\big\lceil \frac{3\cdot 13}{9-3-1}\big\rceil +1=9$ times. Note that the number of recolorings of the vertices of $N(u)$ in $\sigma_{G-u}$ is at most $t=3\cdot 13 +9$. So by taking $G^*=G$, $G^{**}=G-u$, and $|L(u)|=9$ in Lemma~\ref{keylemma}, $\sigma_{G-u}$ can be extended to a recoloring sequence $\sigma_G$ from $\alpha$ to $\beta$ that recolors $u$ at most $\big\lceil \frac{3\cdot 13+9}{9-4-1}\big\rceil +1=13$ times.

\textbf{Configuration $H_2$:} Let $u,v,$ and $w$ be the vertices of a triangle with degrees $4,5,$ and $5$, respectively. By induction, there exists a recoloring sequence $\sigma_{G-\{u,v,w\}}$ from $\alpha|_{G-\{u,v,w\}}$ to $\beta|_{G-\{u,v,w\}}$ that recolors every vertex of $G-\{u,v,w\}$ at most $13$ times. Note that $d_{G-\{u,w\}}(v)=3$ and the total number of recolorings of the vertices of $N(v)\cap (G-\{u,v,w\})$ is at most $t=3\cdot 13$. So by taking $G^*=G-\{u,w\}$, $G^{**}=G-\{u,v,w\}$, and $|L(v)|=9$ in Lemma~\ref{keylemma}, $\sigma_{G-\{u,v,w\}}$ can be extended to a recoloring sequence $\sigma_{G-\{u,w\}}$ from $\alpha|_{G-\{u,w\}}$ to $\beta|_{G-\{u,w\}}$ that recolors $v$ at most $\big\lceil \frac{3\cdot 13}{9-3-1}\big\rceil +1=9$ times. Similarly, by using Lemma~\ref{keylemma}, $\sigma_{G-\{u,w\}}$ can be extended to a recoloring sequence $\sigma_{G-u}$ and then $\sigma_{G-u}$ can be extended to a recoloring sequence $\sigma_{G}$ from $\alpha$ to $\beta$ that recolors $w$ at most $\big\lceil \frac{3\cdot 13+9}{9-4-1}\big\rceil +1=13$ times and $u$ at most $\big\lceil \frac{2\cdot 13+13+9}{9-4-1}\big\rceil +1=13$ times.

\textbf{Configuration $H_3$:} Let $u,v,w$, and $x$ be the vertices of degree $4,5,6$, and $4$, respectively such that $u,v,$ and $w$ are adjacent to each other and $w$ is adjacent to $x$. By induction, there exists a recoloring sequence $\sigma_{G-\{u,v,w,x\}}$ from $\alpha|_{G-\{u,v,w,x\}}$ to $\beta|_{G-\{u,v,w,x\}}$ that recolors every vertex of $G-\{u,v,w,x\}$ at most $13$ times.
Note that $d_{G-\{u,v,x\}}(w)=3$ and hence the total number of recolorings of the vertices of $N(v)\cap (G-\{u,v,w,x\})$ is at most $t=3\cdot 13$. So by taking $G^*=G-\{u,v,x\}$, $G^{**}=G-\{u,v,w,x\}$, and $|L(w)|=9$ in Lemma~\ref{keylemma}, $\sigma_{G-\{u,v,w,x\}}$ can be extended to a recoloring sequence $\sigma_{G-\{u,v,x\}}$ from $\alpha|_{G-\{u,v,x\}}$ to $\beta|_{G-\{u,v,x\}}$ that recolors $w$ at most $\big\lceil \frac{3\cdot 13}{9-3-1}\big\rceil +1=9$ times. Similarly, by using Lemma~\ref{keylemma}, the recoloring sequence $\sigma_{G-\{u,v,x\}}$ can be extended to $\sigma_{G-\{u,x\}}$, then $\sigma_{G-\{u,x\}}$ can be extended to $\sigma_{G-u}$, and then $\sigma_{G-u}$ can be extended to $\sigma_G$ from $\alpha$ to $\beta$ that recolors each of $u,v,$ and $x$ at most $13$ times. 
\end{proof}

\section{Proof of Theorem~\ref{trianglewith6+faces}}
	
Let $\mathcal{G}=\{G\mid $ $G$ is a planar graph and $G$ has a plane embedding such that the faces adjacent to any $3$-face of $G$ are $5^+$-faces and at most three $3$-faces are adjacent to any $5$-face$\}$. Recall that $\mathcal{G}_3= \{ G\mid G$ is a planar graph and $G$ has a plane embedding such that the faces adjacent to any $3$-face of $G$ are $6^+$-faces\}. Note that $\mathcal{G}_3$ is a subclass of $\mathcal{G}$.

For positive integers $x$ and $y$ and a non-negative integer $p$, we use \emph{$x(p)y$-path} to denote a path on $p+2$ vertices, where the intermediate $p$ vertices are of degree $4$ and the endpoints of the path have degree $x$ and $y$ (see Figure~\ref{figure1}(a)). For a face $f$ of a plane graph $G$, a \emph{facial walk} of $f$ is a closed walk along the boundary of $f$. Note that there may be more than one facial walk of a face $f$. By using the discharging method, we prove the following structural result for the graphs in $\mathcal{G}$.

\begin{figure}[hbtp]
\centering
\includegraphics[scale=0.8]{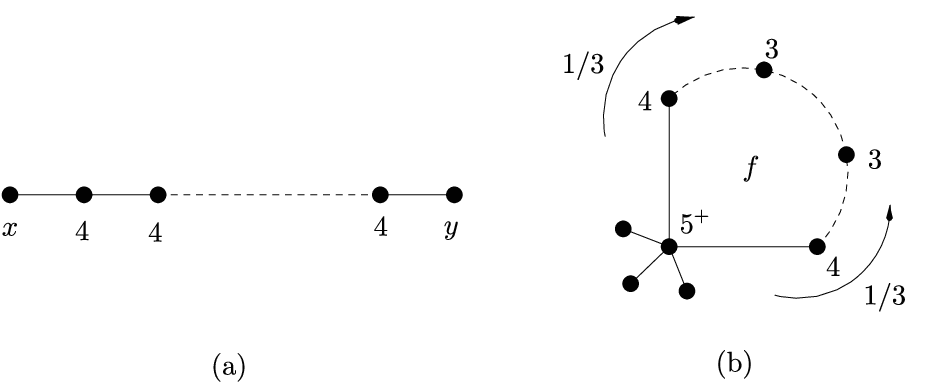}
\caption{The numbers next to the vertices denote their degrees. The graph in (a) is an $x(p)y$-path. In (b), we represent the discharging Rule~$R_1$(c) given in the proof of Theorem~\ref{strtrianglewith5+faces}. Here, $f$ is a $6^-$-face.}\label{figure1}
\end{figure}

\begin{theorem}\label{strtrianglewith5+faces}
If $G\in \mathcal{G}$ and $\delta(G)\geq 3$, then $G$ contains at least one of the following configurations.
\begin{enumerate}
\item[$H_4$:] An edge whose both endpoints are $3$-vertices.

\medskip

\item[$H_5$:] A graph consisting of two paths $P_i=v_1v_2^iv_3^i\ldots v_{p_i+2}^i$ for $i\in \{1,2\}$ such that $P_i$ is a $4(p_i)3$-path with $p_i\leq 3$ for every $i\in \{1,2\}$ and $v_2^1\neq v_2^2$. Note that both the paths share the initial $4$-vertex $v_1$. 

\medskip

\item[$H_6$:] 
A graph consisting of four paths $P_i=v_1v_2^iv_3^i\ldots v_{p_i+2}^i$ for $i\in \{1,2,3,4\}$ such that $P_i$ is a $5(p_i)3$-path with $p_i\leq 4$ for every $i\in \{1,2,3,4\}$ and $v_2^i\neq v_2^j$ for distinct $i,j\in \{1,2,3,4\}$. Note that the paths share the initial vertex $v_1$. Moreover, if any of the $P_i$'s is a $5(4)3$-path, then the $5$-vertex is adjacent to the $3$-vertex of the path making a $5(0)3$-path among the four paths. 
\end{enumerate}
\end{theorem}
\begin{proof}
Let $G\in \mathcal{G}$ and $\delta(G)\geq 3$. Consider a plane embedding of $G$ such that the faces adjacent to any $3$-face of $G$ are $5^+$-faces and at most three $3$-faces are adjacent to any $5$-face. Note that $H_5$ and $H_6$ describe classes of configurations. We call a configuration $H$ to be of \emph{type $H_i$} for $i\in \{5,6\}$ if $H\in H_i$. For the sake of contradiction, assume that $G$ does not contain the configurations $H_4$ and configurations of type $H_5$ and type $H_6$. Since $\delta(G)\geq 3$, two distinct faces cannot share two edges that are consecutive in any facial walks of both the faces. Assign charges to the vertices and the faces of $G$ using the balanced charging method (see Subsection~\ref{balanced}).
Now we re-allocate the charges among the vertices and the faces of $G$ using the following discharging rules. 
\begin{enumerate}
\item[$R_1$:] If $v$ is a $5^+$-vertex, then $v$ sends charge $c=1/3$ through each of its incident edges. Any charge $c$ emerging from $v$ towards an edge $e=vu$ moves with the following rules. 
\begin{enumerate}
	\item $u$ is a $3$-vertex: $c$ gets added to the charge of $u$. 
	\item $u$ is a $5^+$-vertex: $c$ gets added back to the charge of $v$. 
	\item $u$ is a $4$-vertex: If there is no $6^-$-face to which the edge $e$ is incident, then $c$ gets added back to the charge of $v$. Assume that there is a $6^-$face to which the edge $e$ is incident. Note that there are at most two such faces. Along any facial walk of each such $6^-$-face, we will try to find $4(p)3$-paths starting from $u$. If no such $4(p)3$-path exists along any of the facial walks of any of the $6^-$-faces, then $c$ gets added back to the charge of $v$. So assume that there exists such a $4(p)3$-path. By Claim~\ref{one4p3} (stated and proved after the discharging rules), there is only one such $4(p)3$-path, say $P$. Then the charge $c$ gets added to the charge of the last vertex of $P$, that is the $3$-vertex of $P$. Note that the charge $c$ emerges from a $x$-vertex $v$; $x\geq 5$, moves through a $x(p)3$-path $vP$, and gets added to the $3$-vertex of the path. We refer to Figure~\ref{figure1}(b) for a pictorial representation of the Rule~$R_1$(c). We use \emph{long distance discharging} to denote such a discharging by a $x$-vertex $v$; $x\geq 5$, to a $3$-vertex $w$ through a $x(p)3$-path $vv_1v_2\ldots v_kw$ along a facial walk of a $6^-$-face. We call $v$, $w$, $vv_1$, and $v_kw$ as the \emph{source}, \emph{sink}, \emph{source edge}, and \emph{sink edge} of the long distance discharging by $v$, respectively. 
	\end{enumerate}	
\item[$R_2$:] For $5\leq l\leq 6$, if an $l$-face $f$ is adjacent to a $3$-face, then $f$ sends charge $1/3$ to the $3$-face through each common incident edge.
\item[$R_3$:] Every $7$-face $f$ discharges as follows.
\begin{enumerate}
\item If an edge $vu$ is incident to $f$, where $v$ is a $3$-vertex and $u$ is a $4$-vertex, then $f$ sends charge $1/3$ to $v$.

\item If $f$ is adjacent to a $3$-face, then $f$ sends charge $1/3$ to the $3$-face through each common incident edge.
\end{enumerate}

\item[$R_4$:] Every $8^+$-face $f$ discharges as follows.
\begin{enumerate}
\item If a $3$-vertex $v$ is incident to $f$, then $f$ sends charge $1/3$ to $v$. If $f$ contains $v$ several times on its facial walk, then $f$ sends charge $1/3$ to $v$ each time its facial walk visits $v$.
\item If $f$ is adjacent to a $3$-face, then $f$ sends charge $1/3$ to the $3$-face through each common incident edge.
\end{enumerate}

\end{enumerate}

To continue the proof of Theorem~\ref{strtrianglewith5+faces}, we need Claims~\ref{one4p3}-\ref{fi}.

\begin{claim}\label{one4p3}
If $vu$ is an edge, where $v$ is a $5^+$-vertex and $u$ is a $4$-vertex, then there is at most one $4(p)3$-path that has $u$ as the initial vertex and lies along any facial walk of any $6^-$-face incident to $vu$. 
\end{claim}
\begin{proof}[Proof of Claim~\ref{one4p3}]
Let $vu$ be an edge, where $v$ is a $5^+$-vertex and $u$ is a $4$-vertex. Assume that there are at least two $4(p)3$-paths that have $u$ as the initial vertex and lie along some (possibly different) facial walks of the (possibly different) $6^-$-face(s) incident to $vu$. Let two such paths be $P_1$, a $4(p_1)3$-path and $P_2$, a $4(p_2)3$-path for some non-negative integers $p_1$ and $p_2$. Note that $P_1$ and $P_2$ may share some vertices other than $u$. Further note that $P_1$ and $P_2$ may lie on different facial walks of a single $6^-$-face incident to $vu$ or $P_1$ and $P_2$ may lie on different facial walks of distinct $6^-$-faces incident to $vu$. Since $P_1$ and $P_2$ lie on the facial walks of $6^-$-face(s) and these paths cannot contain the $5^+$-vertex $v$ which is also on the facial walks of the $6^-$-face(s), we have $p_1\leq 3$ and $p_2\leq 3$. Then $P_1\cup P_2$ contains a configuration of type $H_5$, a contradiction. So there is only one such $4(p)3$-path. 
\end{proof}

\begin{claim}\label{obs1}
If a source and a source edge of long distance discharging is fixed, then the sink and the sink edge are unique. There may be two sources for a fixed sink and sink edge.
\end{claim}

\begin{proof}[Proof of Claim~\ref{obs1}]
Let $v$ and $vv_1$ be the source and the source edge of long distance discharging. By Claim~\ref{one4p3}, there is a unique $4(p)3$-path starting from $v_1$ along which long distance discharging happens. Therefore, there is a unique sink and a unique sink edge once the source and the source edge are fixed. 
Now let $w$ and $wy$ be the sink and the sink edge of some long distance discharging. It is possible that there are two distinct $6^-$-faces $f_1$ and $f_2$ incident to $wy$. Let $z_1$ and $z_2$ be a $x_1$-vertex and a $x_2$-vertex on $f_1$ and $f_2$, respectively, where $x_1,x_2\geq 5$. Let there be a $x_1(p_1)3$-path $P_1$ and a $x_2(p_2)3$-path $P_2$ from $z_1$ to $w$ and $z_2$ to $w$ in some facial walks of $f_1$ and $f_2$, respectively 
such that both the paths contain the sink edge $wy$. Note that since $f_1$ and $f_2$ are $6^-$-faces, by Rule~$R_1$(c), $z_1$ sends charge $1/3$ to $w$ and $z_2$ sends charge $1/3$ to $w$ through the paths $P_1$ and $P_2$, respectively. So we may conclude that there may be two sources for a fixed sink and sink edge. 
\end{proof}

\begin{claim}\label{fi}
If there is a $6^-$-face incident to an edge $vu$, where $v$ is a $3$-vertex and $u$ is a $4$-vertex, then $v$ gets charge at least $1/3$ through the edge $vu$ by long distance discharging.
\end{claim}
\begin{proof}[Proof of Claim~\ref{fi}]
Let $f$ be a $6^-$-face incident to an edge $vu$, where $v$ is a $3$-vertex and $u$ is a $4$-vertex. By taking the direction of traversal from $v$ to $u$, we search for a $3(p)x$-path along a facial walk of $f$ such that the initial vertex of the path is the $3$-vertex $v$ and $x\geq 5$. If such a path does not exist, then along that facial walk of $f$, there exists a $3(p)3$-path $P$ for some non-negative integer $p$ or a cycle $C$ such that exactly one of its vertices is a $3$-vertex and the remaining all others are $4$-vertices. Note that we can find a suitable $4$-vertex, say $v^*$ on $P$ and $C$ such that $P$ and $C$ can be viewed as the union of two paths $P_1$ and $P_2$, where $P_i$ is a $4(p_i)3$-path with $p_i\leq 3$ for $i\in\{1,2\}$, $P_1$ and $P_2$ have $v^*$ as the initial $4$-vertex, and the neighbors of $v^*$ in $P_1$ and $P_2$ are different. So $P$ and $C$ are configurations of type $H_5$, a contradiction. Hence there exists a $3(p)x$-path $P$ along a facial walk of $f$ such that the initial vertex of the path is the $3$-vertex $v$ and $x\geq 5$. Let $w$ be the $x$-vertex of $P$ and $w'$ be the vertex adjacent to $w$ in $P$. Now since $w$ is a $5^+$-vertex and $w'$ is a $4$-vertex, by Rule~$R_1$(c) for the vertex $w$, there will be a long distance discharging with $w$ and $ww'$ as the source and the source edge and $v$ and $uv$ as the sink and the sink edge. 
\end{proof}

We now return to the proof of Theorem~\ref{strtrianglewith5+faces}. We show that at the end of the discharging process, every vertex and every face has non-negative charge and obtain a contradiction to the fact that $ch(G)=-8$.\medskip

\textbf{$3$-vertex:} 
Let $v$ be a 3-vertex of $G$ and $v_1,v_2$, and $v_3$ be its neighbors. Initially, $v$ had charge $-1$. Since $G$ does not contain the configuration $H_4$ and $\delta(G)\geq 3$, $v_i$ is a $4^+$-vertex for every $i\in \{1,2,3\}$. For every $i\in \{1,2,3\}$, let $f_i$ be the face incident to the edge $vv_i$ such that $f_i$ appears next to the edge $vv_i$ when we move clockwise around $v$. Note that $f_i$'s may not be distinct. We show that $v$ gets sufficient charge from its neighbors, its incident faces, and some $5^+$-vertices (by long distance discharging) to make its charge non-negative.

First assume that $f_i$ is a distinct face for every $i\in\{1,2,3\}$. Consider the pairs $\{v_i,f_i\}$ for $i\in \{1,2,3\}$. Note that the pairs are pairwise non-intersecting. Now we show that due to each pair $\{v_i,f_i\}$; $i\in \{1,2,3\}$, $v$ gets charge $1/3$. This happens as follows.  If $v_i$ is a $5^+$-vertex, then $v_i$ sends charge $1/3$ to $v$ by Rule~$R_1$(a). So assume that $v_i$ is a $4$-vertex. If $f_i$ is a $7^+$-face, then, since $v_i$ is a $4$-vertex, $v$ gets charge $1/3$ from $f_i$ by Rule~$R_3$(a) or $R_4$(a). So assume that $f_i$ is a $6^-$-face. Then by Claim~\ref{fi}, $v$ gets charge at least $1/3$ through the edge $vv_i$ by long distance discharging from some source. Since there are three non-intersecting pairs $\{v_i,f_i\}$; $i\in \{1,2,3\}$, $v$ gets charge $1/3$ at least three times. So in total $v$ gets charge at least $3\cdot (1/3)$ and hence $v$ has non-negative charge at the end of the discharging process. 
 
Now assume that at least two faces among $f_1$, $f_2$, and $f_3$ are same. Without loss of generality, assume that $f_1=f_2$. Recall that all the neighbors of $v$ are $4^+$-vertices. Since $f_1=f_2$, by the way the faces $f_1$ and $f_2$ are defined, the edge $vv_2$ has to be a cut edge and gets counted twice on any facial walk of $f_1$. Note that any facial walk of $f_1$ would look like $vv_2u'\ldots w'v_2vv_3\ldots v_1v$, where $u'$ and $w'$ are distinct neighbors of $v_2$ different from $v$. Hence it is easy to verify that $f_1$ is a $8^+$-face.Now we consider two cases depending on $f_1=f_3$ or $f_1\neq f_3$. First assume that $f_1=f_3$ (all three faces are same). Then $vv_1$, $vv_2$, and $vv_3$ are cut edges and hence $f_1$ contains $v$ thrice on any of its facial walk. So $v$ gets charge $3\cdot (1/3)$ from $f_1$ by Rule~$R_4$(a). Hence $v$ gets sufficient charge to make its charge non-negative. Now assume that $f_1\neq f_3$. Then $f_1$ contains $v$ twice on any of its facial walk. Hence $v$ gets charge $2\cdot (1/3)$ from $f_1$ by Rule~$R_4$(a). We now show that $v$ gets another $1/3$ charge due to the pair $\{v_3,f_3\}$. If $v_3$ is a $5^+$-vertex, then $v$ gets charge $1/3$ from $v_3$ by Rule~$R_1$(a). So $v_3$ is a $4$-vertex. Now if $f_3$ is $7^+$-face,  then, since $v_3$ is a $4$-vertex, $v$ gets charge $1/3$ from $f_3$ by Rule~$R_3$(a) or $R_4$(a). So assume that $f_3$ is a $6^-$-face. Then by Claim~\ref{fi}, $v$ gets charge $1/3$ through the edge $vv_3$ by long distance discharging. Hence $v$ has non-negative charge at the end of the discharging process.

\textbf{$4$-vertex:} Initially, every $4$-vertex had charge $0$. Since any $4$-vertex neither gains nor loses any charge during the discharging process, its charge remains the same, that is $0$.

\textbf{$5$-vertex:} Let $v$ be a $5$-vertex. Initially, $v$ had charge $1$. So $v$ can afford to lose charge at most $1$ during the discharging process. Recall that through each edge $e$ incident to $v$, $v$ loses charge either $0$ or $1/3$ by Rule~$R_1$. We show that $v$ loses charge $1/3$ through at most three of its incident edges implying that $v$ has non-negative charge at the end of the discharging process. In particular, we show that if $v$ loses charge $1/3$ through at least four of its incident edges, then a configuration of type $H_5$ is obtained, a contradiction. For this purpose, it is sufficient to show that if $v$ loses charge $1/3$ through an incident edge $vu$, then a $5(p)3$-path starting with the edge $vu$ is obtained with $p\leq 4$. Moreover, if $p=4$ for any such path, then we have to show that $v$ is adjacent to the last vertex of the $5(p)3$-path. 

Let $vu$ be an edge and $v$ loses charge $1/3$ through the edge $vu$. Note that $v$ does not lose any charge through $vu$ if $u$ is a $5^+$-vertex by Rule~$R_1$(b). So $u$ is a $3$-vertex or a $4$-vertex since $\delta(G)\geq 3$. First assume that $u$ is a $3$-vertex. Then the charge $1/3$ gets added to the charge of $u$ by Rule~$R_1(a)$ and $vu$ is a $5(0)3$-path. Now assume that $u$ is a $4$-vertex. Then $v$ loses the charge by Rule~$R_1$(c), that is by long distance discharging. Since the source $v$ and the source edge $vu$ are fixed, by Claim~\ref{obs1}, there is a unique sink and a unique sink edge. The sink is a $3$-vertex $w$ that is incident to a $6^-$-face $f$ which is also incident to $v$. Since this long distance discharging requires a $5(p)3$-path $P$ from $v$ to $w$ along a facial walk of $f$ that is a $6^-$-face, $p\leq 4$. Moreover, if $p=4$, then $f$ has to be a $6$-face and $Pv$ has to be a facial walk of $f$; thus $v$ is adjacent to the $3$-vertex of $P$. So we get the desired $5(p)3$-path. 

\textbf{$6^+$-vertex:} Let $v$ be a $6^+$-vertex. Initially, $v$ had charge $d(v)-4$. Since any $6^+$-vertex loses charge $1/3$ along each of the edges incident to it by Rule~$R_1$, $v$ loses charge at most $(1/3)\cdot d(v)$. Since $d(v)\geq 6$, we have $d(v)-4-d(v)/3\geq 0$. Hence $v$ has non-negative charge at the end of the discharging process.

\textbf{$4$-face:} Initially, every $4$-face has charge $0$. Since any $4$-face neither gains nor loses any charge during the discharging process, its charge remains the same, that is 0.

\textbf{$k$-face for $k\in\{5,6\}$:} Let $f$ be a $5$-face. Initially, $f$ had charge $1$. Since $G\in \mathcal{G}$, there are at most three $3$-faces adjacent to $f$. Note that $f$ sends charge $1/3$ to each of its adjacent $3$-faces through the common incident edges by Rule~$R_2$. Since $\delta(G)\geq 3$ and $f$ is a $5$-face, any $3$-face cannot share more than one edge with $f$. This implies that $f$ loses charge at most $3\cdot(1/3)=1$ and hence $f$ has non-negative charge at the end of the discharging process. Now let $f$ be a $6$-face. Initially, $f$ had charge $2$. Since $f$ sends charge $1/3$ to each of its adjacent $3$-faces through the common incident edges by Rule~$R_2$, $f$ loses charge at most $6\cdot(1/3)=2$. Hence $f$ has non-negative charge at the end of the discharging process.

\textbf{$7$-face:} Let $f$ be a $7$-face. Initially, $f$ had charge $3$. Since $f$ sends charge $1/3$ to each of its adjacent $3$-faces through the common incident edges by Rule~$R_3$(b), $f$ loses charge at most $7\cdot(1/3)$ to the $3$-faces adjacent to it. Let $S_G(f)=\{v\in V(G)\mid v$ is a $3$-vertex and $v$ has a $4$-neighbor $u$ such that $vu$ is incident to $f\}$. We claim that $|S_G(f)|\leq 2$.  If $|S_G(f)|\geq 3$, then let $v_i\in S_G(f)$ for $i\in \{1,2,3\}$. Then for each $i\in \{1,2,3\}$, $v_i$ has a $4$-neighbor $u_i$ such that $v_iu_i$ is incident to $f$. Since $G$ does not contain the configuration $H_4$, $v_i$ is not adjacent to $v_j$ for any $i,j\in \{1,2,3\}$. Since $f$ is a $7$-face, without loss of generality, we may assume that $v_1v_1'v_2v_2'v_3v_4v_5v_1$ is a facial walk of $f$ for some vertices $v_1',v_2',v_4,$ and $v_5$. Note that $u_2=v_1'$ or $u_2=v_2'$. Then $\{v_1,v_1',v_2\}$ or $\{v_2,v_2',v_3\}$ induces a configuration of type $H_5$, a contradiction. 
So $|S_G(f)|\leq 2$. Recall that $f$ sends charge $1/3$ to each vertex of $S_G(f)$ by Rule~$R_3$(a). So $f$ loses charge at most $2\cdot (1/3)$ to the vertices of $S_G(f)$. This implies that in total $f$ loses charge at most $9\cdot(1/3)=3$ and hence $f$ has non-negative charge at the end of the discharging process.

\textbf{$8^+$-face:} Let $f$ be a $8^+$-face. Initially, $f$ had charge $d(f)-4$. Since $f$ sends charge $1/3$ to each of its adjacent $3$-faces through the common incident edges by Rule~$R_4$(b), $f$ loses charge at most $\frac{1}{3}\cdot d(f)$ to the $3$-faces adjacent to it. Since $G$ does not contain the configuration $H_4$, there are at most $\big\lfloor\frac{d(f)}{2}\big\rfloor$ $3$-vertices incident to $f$. So $v$ loses charge at most $\frac{1}{3}\cdot\big\lfloor \frac{d(f)}{2} \big\rfloor$ to the $3$-vertices incident to it by Rule~$R_4$(a). Hence in total $f$ loses charge at most $\frac{1}{3}\cdot(d(f)+\big\lfloor \frac{d(f)}{3} \big\rfloor)$. Since $d(f)\geq 8$, we have $\frac{1}{3}\cdot (d(f)+\big\lfloor \frac{d(f)}{3} \big\rfloor)\leq d(f)-4$. So $f$ has non-negative charge at the end of the discharging process.
\end{proof}

In the following observations, we describe the structure of the minimum order configuration of type $H_5$ and type $H_6$ in a graph.

\begin{obs}\label{obsH5}
Let a graph $G$ contain the configuration of type $H_5$ and $H^*$ be the minimum order subgraph of $G$ of type $H_5$. Let $P_i=v_1v_2^i\ldots v_{p_i+2}^i$; $i\in \{1,2\}$ be the $4(p_i)3$-path such that $H^*=P_1\cup P_2$, and $|V(P_1)|\geq |V(P_2)|$. Then one of the following holds.
\begin{enumerate}
\item $P_1$ and $P_2$ do not share any vertex other than the initial $4$-vertex $v_1$.
\item If $P_1$ and $P_2$ share a vertex $w$ other than the initial $4$-vertex $v_1$, then there is a $4(p_1^*)3$-path $P_1^*$ for some $p_1^*<p_1\leq 3$ in $G-w$ which is a subpath of $P_1$ starting with $v_1$ such that $P_1^*$ and $P_2$ do not share any vertex other than initial $4$-vertex $v_1$ and $V(H^{*})=V(P_1^*\cup P_2)$.
\end{enumerate}\end{obs}
\begin{proof}
 If $P_1$ and $P_2$ do not share any vertex other than the initial $4$-vertex $v_1$, then we are done. So assume that $P_1$ and $P_2$ share a vertex $w$, say $v_i^1$ other than the initial $4$-vertex $v_1$. 
 Then all the vertices appearing after $w$ in the paths $P_1$ and $P_2$ are shared. Otherwise, we can remove some non-shared vertices that appears after $w$ in the path $P_1$ or $P_2$ to get a configuration of type $H_5$ with less order. We now show that $i\geq 3$. Assume that $i=2$. By the definition of $H_5$, $v_2^1\neq v_2^2$. So $w=v_2^1=v_j^2$ for some $j\geq 3$. Then, since the vertices that appears after $w$ in the paths $P_1$ and $P_2$ are shared, $|V(P_2)|\geq |V(P_1)|$, a contradiction. So we have $i\geq 3$. 
  Let $P_1^*=v_1\ldots v_{i-1}^1$. Note that the path $P_1^*w$ is chordless since $H^*$ is of minimum order. Then $P_1^*$ is a $4(p_1^*)3$-subpath of $P_1$ for some $p_1^*<p_1\leq 3$ in $G-w$.  Clearly, $P_1^*$ and $P_2$ do not share any vertex other than initial $4$-vertex $v_1$ and $V(H^{*})=V(P_1^*\cup P_2)$.
\end{proof}

\begin{obs}\label{obsH6}
Let a graph $G$ contain the configuration of type $H_6$ and $H^{**}$ be the minimum order subgraph of $G$ of type $H_6$. Let $P_i=v_1v_2^i\ldots v_{p_i+2}^i$; $i\in \{1,2,3,4\}$ be the $5(p_i)3$-path such that $H^{**}=P_1\cup P_2\cup P_3\cup P_4$ and $|V(P_1)|\geq |V(P_2)|\geq |V(P_3)|\geq |V(P_4)|$. Then for every $i\in \{1,2,3\}$, one of the following holds.
\begin{enumerate}
\item $P_i$ and $P_{i+1}\cup \ldots \cup P_4$ do not share any vertex other than the initial $5$-vertex $v_1$.
\item If $P_i$ and $P_{i+1}\cup \ldots \cup P_4$ share a vertex $w$ other than the initial $5$-vertex $v_1$, then there is a $5(p_i^*)3$-path $P_i^*$ for some $p_i^*<p_i\leq 4$ in $G-w$ which is a subpath of $P_i$ starting with $v_1$ such that $P_i^*$ does not share any vertex with $P_{i+1}\cup \ldots \cup P_4$ and $V(H^{**})=V(P_i^*)\cup V((P_1\cup\ldots\cup P_4)-P_i)$
\end{enumerate}

\end{obs}
\begin{proof}
Fix $i\in \{1,2,3\}$. If $P_i$ and $P_{i+1}\cup \ldots \cup P_4$ do not share any vertex other than the initial $5$-vertex $v_1$, then we are done. So assume that $P_i$ and $P_{i+1}\cup \ldots \cup P_4$ share a vertex $w$, say $v_k^i$ with least index $k$ other than $1$. Let $P_j$ be the path for some $j\in \{i+1,\ldots,4\}$ that contains $w$. 
Then all the vertices appearing after $w$ in $P_i$ and $P_j$ are shared. Otherwise, we can remove some non-shared vertices that appears after $w$ in the path $P_i$ or $P_j$ to get a configuration of type $H_6$ with less order. We now show that $k\geq 3$. Assume that $k=2$. By the definition of $H_6$, $v_2^i\neq v_2^j$. So $w=v_2^i=v_l^j$ for some $l\geq 3$. Then, since the vertices that appears after $w$ in $P_i$ and $P_j$ are shared, $|V(P_j)|\geq |V(P_i)|$, a contradiction. So we have $k\geq 3$. Let $P_i^*=v_1v_2^i\ldots v_{k-1}^i$. Note that the path $P_i^*w$ is chordless since $H^{**}$ is or minimum order. Then $P_i^*$ is a $5(p_i^*)3$-subpath of $P_i$ for some $p_i^*<p_i\leq 4$ in $G-w$. Clearly, $P_i^*$ and $P_{i+1}\cup\ldots \cup P_4$ do not share any vertex other than the initial $5$-vertex $v_1$ and $V(H^{**})=V(P_i^*)\cup V((P_1\cup\ldots\cup P_4)-P_i)$. 
\end{proof}

Recall that $\mathcal{G}_3$ is a subclass of $\mathcal{G}$. To prove Theorem~\ref{trianglewith6+faces}, we prove the following stronger result.

\begin{theorem}\label{trianglewith5+faces}
If $G\in \mathcal{G}$, then diam$(\mathcal{G}(G,L))\leq 242|V(G)|$.
\end{theorem} 
\begin{proof}
Let $G\in \mathcal{G}$ and $\alpha$ and $\beta$ be any two $L$-colorings of $G$. In order to prove that diam$(\mathcal{G}(G,L))\leq 242|V(G)|$, we prove the claim that there exists a recoloring sequence from $\alpha$ to $\beta$ that recolors every vertex of $G$ at most $242$ times. We use induction on the number of vertices of $G$ to prove the claim. If $|V(G)|=1$, then the claim is trivial. So the base case is true. First assume that there is a $2^-$-vertex $v\in V(G)$. By induction, there is a recoloring sequence $\sigma_{G-v}$ from $\alpha|_{G-v}$ to $\beta|_{G-v}$ that recolors every vertex of $G-v$ at most $242$ times. Note that the total number of recolorings of the vertices of $N(v)$ in $\sigma_{G-v}$ is at most $t=d(v)\cdot 242$. So by taking $G^*=G$, $G^{**}=G-v$, and $|L(v)|=7$ in Lemma~\ref{keylemma}, the sequence $\sigma_{G-v}$ can be extended to a recoloring sequence $\sigma_G$ from $\alpha$ to $\beta$ that recolors $v$ at most $\big\lceil \frac{d(v)\cdot 242}{7-d(v)-1}\big\rceil +1\leq 122<242$ times since $d(v)\leq 2$. Similarly, we can show that $\sigma_{G-v}$ can be extended to a desired recoloring sequence if $d(v)<2$. Now we may assume that $\delta(G)\geq 3$. Then by Theorem~\ref{strtrianglewith5+faces}, $G$ contains at least one of the configurations $H_4$, $H_5$, and $H_6$. Now for each configuration, we show that the claim holds.

\textbf{Configuration $H_4$:} Let $uv$ be an edge such that $u$ and $v$ are $3$-vertices. By induction, there exists a recoloring sequence $\sigma_{G-\{u,v\}}$ from $\alpha|_{G-\{u,v\}}$ to $\beta|_{G-\{u,v\}}$ that recolors every vertex of $G-\{u,v\}$ at most $242$ times. Note that $d_{G-u}(v)=2$ and the total number of recolorings of the vertices of $N(v)\cap (G-\{u,v\})$ in $\sigma_{G-\{u,v\}}$ is at most $2\cdot 242$. So by taking $G^*=G-u$, $G^{**}=G-\{u,v\}$, and $|L(v)|=7$ in Lemma~\ref{keylemma}, $\sigma_{G-\{u,v\}}$ can be extended to a recoloring sequence $\sigma_{G-u}$ from $\alpha|_{G-u}$ to $\beta|_{G-u}$ that recolors $v$ at most $\big\lceil \frac{2\cdot 242}{7-2-1}\big\rceil +1=122<242$ times. Note that the total number of recoloring of the vertices of $N(u)$ in $\sigma_{G-u}$ is at most $t=2\cdot 242+122$. So by taking $G^*=G$, $G^{**}=G-u$, and $|L(v)|=7$ in Lemma~\ref{keylemma}, $\sigma_{G-u}$ can be extended to a recoloring sequence $\sigma_{G}$ from $\alpha$ to $\beta$ that recolors $u$ at most $\big\lceil \frac{2\cdot 242+122}{7-3-1}\big\rceil+1=203<242$ times.

\begin{claim}\label{triangleplanarcl1}
If $v_1v_2\ldots  v_{p+2}$ is a $x(p)3$-path of a subgraph $G'$ of $G$ and a subgraph $G''$ of $G'$ contains that path 
and there is a recoloring sequence $\sigma_{G''-\{v_2,\ldots,v_{p+2}\}}$ from $\alpha|_{G''-\{v_2,\ldots,v_{p+2}\}}$ to $\beta|_{G''-\{v_2,\ldots,v_{p+2}\}}$ that recolors every vertex of $G''-\{v_1,\ldots,v_{p+2}\}$ at most $242$ times and $v_1$ at most $c_1$ times, then $\sigma_{G''-\{v_2,\ldots,v_{p+2}\}}$ can be extended to a recoloring sequence $\sigma_{G''}$ from $\alpha|_{G''}$ to $\beta|_{G''}$ that recolors $v_i$ at most $c_i$ times, where $c_i=\big\lceil \frac{484+c_{i-1}}{3}\big\rceil +1$ for $i\in\{2,\ldots,p+2\}$.
\end{claim}
\begin{proof}[Proof of Claim~\ref{triangleplanarcl1}]
Let the hypothesis of the claim be true. Consider the sets $X_i=\{v_i,\ldots,v_{p+2}\}$ for $i\in \{2,\ldots,p+2\}$. Note that $d_{G'}(v_2)=4$, $d_{G''}(v_2)\leq 4$, and $d_{G''-X_3}(v_2)\leq 3$. We take $d_{G''}(v_2)=4$ and $d_{G''-X_3}(v_2)=3$ to determine $c_2$. For other values of $d_{G''}(v_2)$ and $d_{G''-X_3}(v_2)$, $c_2$ can be found to be even less by a similar approach. Note that the total number of recolorings of the vertices of $N_{G''}(v_2)\cap (G''-X_2)$ in $\sigma_{G''-X_2}$ is at most $t=2\cdot 242+c_1$. So by taking $G^*=G''-X_3$, $G^{**}=G''-X_2$, and $|L(v)|=7$ in Lemma~\ref{keylemma}, the recoloring sequence $\sigma_{G''-X_2}$ can be extended to a recoloring sequence $\sigma_{G''-X_3}$ from $\alpha|_{G''-X_3}$ to $\beta|_{G''-X_3}$ that recolors $v_2$ at most $c_2= \big\lceil \frac{2\cdot 242+c_1}{7-3-1}\big\rceil +1= \big\lceil \frac{484+c_1}{3}\big\rceil +1$ times. Similarly, by using Lemma~\ref{keylemma}, $\sigma_{G''-X_3}$ can be extended to a recoloring sequence $\sigma_{G''-X_4}$ that recolors $v_3$ at most $c_3= \big\lceil \frac{484+c_2}{3}\big\rceil +1$ times, then $\sigma_{G''-X_4}$ to $\sigma_{G''-X_5}$ and so on. At the end, we obtain a desired recoloring sequence. 
\end{proof}

\textbf{Configuration $H_5$:} Without loss of generality, assume that $H^*$ is a subgraph of $G$ that has minimum possible number of vertices among all the configurations of type $H_5$ in $G$. Let $P_1=v_1v_2^1\ldots v_{p_1+2}^1$ and $P_2=v_1v_2^2\ldots v_{p_2+2}^2$ be the $4(p_1)3$-path and the $4(p_2)3$-path, respectively with $|V(P_1)|\geq |V(P_2)|$ that forms the subgraph $H^*$. Now we define $P_1^*$ and $p_1^*$ with the help of Observation~\ref{obsH5}. If Observation~\ref{obsH5}(a) holds, that is $P_1$ and $P_2$ do not share any vertex other than $v_1$, then we take $P_1^*=P_1$ and $p_1^*=p_1$. Now if Observation~\ref{obsH5}(b) holds, then we define $P_1^*$ and $p_1^*$ as defined in Observation~\ref{obsH5}(b). 

Let $P=V(P_1^*)\cup V(P_2)$, $P'=P\setminus \{v_1\}$, and $P''=P'\setminus V(P_1^*)$.
 By induction, there exists a recoloring sequence $\sigma_{G-P}$ from $\alpha|_{G-P}$ to $\beta|_{G-P}$ that recolors every vertex of $G-P$ at most $242$ times. Note that $d_{G-P'}(v_1)=2$ and the total number of recolorings of the vertices of $N(v_1)\cap (G-P)$ in $\sigma_{G-P}$ is at most $t=2\cdot 242$. Hence by taking $G^*=G-P'$, $G^{**}=G-P$, and $|L(v)|=7$ in Lemma~\ref{keylemma}, $\sigma_{G-P}$ can be extended to a recoloring sequence $\sigma_{G-P'}$ from $\alpha|_{G-P'}$ to $\beta|_{G-P'}$ that recolors $v_1$ at most $\big\lceil \frac{2\cdot 242}{7-2-1}\big\rceil +1=122<242$ times. Note that $P_1^*$ is contained in $G-P''$ and it is a $4(p_1^*)3$-path of a subgraph $G'$ of $G$ containing $G-P''$, where $G'=G$ if $P_1^*=P_1$ and $G'=G-w$ if $P_1^*$ is a proper subpath of $P_1$ (see Observation~\ref{obsH5}(b)).  Further note that $\sigma_{G-P'}$ is a recoloring sequence that recolors every vertex of $G-P'$ at most $242$ times and $v_1$ at most $122$ times. Hence by taking $G''=G-P''$ in Claim~\ref{triangleplanarcl1}, $\sigma_{G-P'}$ can be extended to a recoloring sequence $\sigma_{G-P''}$ from $\alpha|_{G-P''}$ to $\beta|_{G-P''}$ that recolors $v_i^1$ at most $c_i$ times, where $c_i=\big\lceil \frac{484+c_{i-1}}{3}\big\rceil +1$ for $i\in\{2,\ldots,p_1^*+2\}$. Similarly, by Claim~\ref{triangleplanarcl1}, $\sigma_{G-P''}$ can be extended to a recoloring sequence $\sigma_{G}$ from $\alpha$ to $\beta$ that recolors $v_i^2$ at most $c_i$ times, where $c_i=\big\lceil \frac{484+c_{i-1}}{3}\big\rceil +1$ for $i\in\{2,\ldots,p_2+2\}$. Let $p=\max\{p_1^*,p_2\}$. Note that $p\leq 3$. Now since $c_1=122$ and $p\leq 3$, we have $\displaystyle\max_{2\leq i\leq p+2}\{c_i\}\leq 242$. So $\sigma_G$ is the desired recoloring sequence.

\textbf{Configuration $H_6$:} Without loss of generality, assume that $H^{**}$ is a subgraph of $G$ that has minimum possible number of vertices among all the configurations of type $H_6$ in $G$. Let $P_1=v_1v_2^1\ldots v_{p_1+2}^1$, $P_2=v_1v_2^2\ldots v_{p_2+2}^2$, $P_3=v_1v_2^3\ldots v_{p_3+2}^3$, and $P_4=v_1v_2^4\ldots v_{p_4+2}^4$ be the $5(p_1)3$-path, the $5(p_2)3$-path, the $5(p_3)3$-path, and the $5(p_4)3$-path, respectively with $|V(P_1)|\geq |V(P_2)|\geq |V(P_3)|\geq |V(P_4)|$ that forms the subgraph $H^{**}$. Note that $4\geq p_1\geq p_2\geq p_3\geq p_4$. So if $p_4=4$, then $p_i=4$ for every $i\in \{1,2,3,4\}$. This contradicts the definition of $H_6$ stating the existence of a $5(0)3$-path among the four paths. So $p_4\leq 3$. Now for $i=1$ to $3$ in order, we define $P_i^*$ and $p_i^*$ with the help of Observation~\ref{obsH6}. If Observation~\ref{obsH6}(a) holds, that is $P_i$ does not share any vertex with $P_{i+1}\cup \ldots\cup P_4$, then we take $P_i^*=P_i$ and $p_i^*=p_i$. If Observation~\ref{obsH6}(b) holds, then we define $P_i^*$ and $p_i^*$ as defined in Observation~\ref{obsH6}(b). Clearly, $V(H^{**})=V(P_1^*\cup P_2^*\cup P_3^*\cup P_4)$. Now we show that $p_i^*\leq 3$ for every $i\in \{1,2,3\}$. If $p_i\leq 3$, then we are done. So assume that $p_i=4$, that is $P_i$ is a $5(4)3$-path. Then by the definition of $H_6$, there is a $5(0)3$-path $P_j$ for some $j\in\{i+1,\ldots,4\}$ such that $P_i$ and $P_j$ shares a vertex. Then $P_i^*$ is a proper subpath of $P_i$ by Observation~\ref{obsH6}(b) and hence $p_i^*\leq 3$.

Let $P=\left(\cup_{1\leq j\leq 3}(V(P_j^*)\right)\cup V(P_4)$, $P'=P\setminus \{v_1\}$, $Q_1=P'\setminus V(P_1^*)$, $Q_2=Q_1\setminus V(P_2^*)$, and $Q_3=Q_2\setminus V(P_3^*)$. By induction, there exists a recoloring sequence $\sigma_{G-P}$ from $\alpha|_{G-P}$ to $\beta|_{G-P}$ that recolors every vertex of $G-P$ at most $242$ times. Note that $d_{G-P'}(v_1)=1$ and the total number of recolorings of the vertices of $N(v_1)\cap (G-P)$ in $\sigma_{G-P}$ is at most $t=242$. So by taking $G^*=G-P'$, $G^{**}=G-P$, and $|L(v)|=7$ in Lemma~\ref{keylemma}, $\sigma_{G-P}$ can be extended to a recoloring sequence $\sigma_{G-P'}$ from $\alpha|_{G-P'}$ to $\beta|_{G-P'}$ that recolors $v_1$ at most $c_1=\big\lceil \frac{242}{7-1-1}\big\rceil +1=50$ times. Note that $P_1^*$ is contained in $G-Q_1$ and $P_1^*$ is a $5(p_1^*)3$-path of a subgraph $G'$ of $G$ containing $G-Q_1$, where $G'=G$ if $P_1^*=P_1$ and $G'=G-w$ if $P_1^*$ is a proper subpath of $P_1$ (see Observation~\ref{obsH6}(b)). 
Further note that $\sigma_{G-P'}$ recolors every vertex of $G-Q_1$ at most $242$-times and $v_1$ at most $50$ times. So by taking $G''=G-Q_1$ in Claim~\ref{triangleplanarcl1}, $\sigma_{G-P'}$ can be extended to a recoloring sequence $\sigma_{G-Q_1}$ such that $v_i^1$ is recolored at most $c_i$ times, where $c_i=\big\lceil\frac{484+c_{i-1}}{3}\big\rceil+1$ for $i\in \{2,\ldots,p_1^*+2\}$. Now since $c_1=50$ and $p_1^*\leq 3$, we have $\displaystyle\max_{1\leq i\leq p_1^*+2}\{c_i\}\leq 242$. Similarly, we use Claim~\ref{triangleplanarcl1} repeatedly to extend the recoloring sequence $\sigma_{G-Q_1}$ to $\sigma_{G-Q_2}$, $\sigma_{G-Q_2}$ to $\sigma_{G-Q_3}$, and lastly, $\sigma_{G-Q_3}$ to $\sigma_{G}$ such that $\sigma_G$ is the desired recoloring sequence.
\end{proof}

\section{Proof of Theorem~\ref{planarwithout4cycles}}

We use the following structural result given by Shen et al.~\cite{shen} to prove Theorem~\ref{planarwithout4cycles}.

\begin{lemma}[\cite{shen}]\label{strplanarwithout4cycles}
Let $G$ be a plane graph such that $\delta(G)\geq 4$ and $G$ has no $4$-cycles. Then there exists a $4$-vertex $v$ in $G$ that is incident to two non-adjacent $3$-faces, say $f_1$ induced by $\{v,v_1,v_2\}$ and $f_2$ induced by $\{v,v_3,v_4\}$, such that every vertex but at most one of $\{v_1, v_2, v_3, v_4\}$ has degree exactly~$4$.
\end{lemma}

\begin{proof}[Proof of Theorem~\ref{planarwithout4cycles}]
Let $\alpha$ and $\beta$ be two $L$-colorings of $G$. In order to prove that diam$(\mathcal{G}(G,L))\leq 29|V(G)|$, we prove the claim that there exists a recoloring sequence from $\alpha$ to $\beta$ that recolors every vertex of $G$ at most $29$ times. We use induction on the number of vertices to prove the claim. If $|V(G)|=1$, then the claim is trivial. So the base case is true. First assume that there is a $3^-$-vertex $v\in V(G)$. By induction, there exists a recoloring sequence $\sigma_{G-v}$ from $\alpha|_{G-v}$ to $\beta|_{G-v}$ that recolors every vertex of $G-v$ at most $29$ times. 
Note that the total number of recolorings of the vertices of $N(v)$ in $\sigma_{G-v}$ is at most $t=d(v)\cdot 29$. So by taking $G^*=G$, $G^{**}=G-v$, and $|L(v)|=8$ in Lemma~\ref{keylemma}, $\sigma_{G-v}$ can be extended to a recoloring sequence $\sigma_G$ from $\alpha$ to $\beta$ that recolors $v$ at most $\big\lceil \frac{d(v)\cdot 29}{8-d(v)-1}\big\rceil +1=23<29$ times since $d(v)\leq 3$. 
So we may assume that $\delta(G)\geq 4$. Then by Lemma~\ref{strplanarwithout4cycles}, there exists a $3$-face $f$ induced by $\{v_1,v_2,v_3\}$ such that $d(v_i)=4$ for $i\in \{1,2,3\}$ and $v_1$ has a $4$-neighbor $v_4$ that is not incident to $f$. By induction, there exists a recoloring sequence $\sigma_{G-\{v_1,v_2,v_3,v_4\}}$ from $\alpha|_{G-\{v_1,v_2,v_3,v_4\}}$ to $\beta|_{G-\{v_1,v_2,v_3,v_4\}}$ that recolors every vertex of $G-\{v_1,v_2,v_3,v_4\}$ at most $29$ times. 
Note that $d_{G-\{v_1,v_3,v_4\}}(v_2)=2$ and the total number of recolorings of the vertices of $N(v_2)\cap (G-\{v_1,v_2,v_3,v_4\})$ in $\sigma_{G-\{v_1,v_2,v_3,v_4\}}$ is at most $t=2\cdot 29$. So by taking $G^*=G-\{v_1,v_3,v_4\}$, $G^{**}=G-\{v_1,v_2,v_3,v_4\}$, and $|L(v_2)|=8$ in Lemma~\ref{keylemma}, $\sigma_{G-\{v_1,v_2,v_3,v_4\}}$ can be extended to a recoloring sequence $\sigma_{G-\{v_1,v_3,v_4\}}$ from $\alpha|_{G-\{v_1,v_3,v_4\}}$ to $\beta|_{G-\{v_1,v_3,v_4\}}$ that recolors $v_2$ at most $\big\lceil \frac{2\cdot 29}{8-2-1}\big\rceil +1=13$ times. Note that $d_{G-\{v_1,v_4\}}(v_3)=3$ and the total number of recolorings of the vertices of $N(v_3)\cap (G-\{v_1,v_3,v_4\})$ in $\sigma_{G-\{v_1,v_3,v_4\}}$ is at most $t=2\cdot 29+13$. So by taking $G^*=G-\{v_1,v_4\}$, $G^{**}=G-\{v_1,v_3,v_4\}$, and $|L(v_3)|=8$ in Lemma~\ref{keylemma}, $\sigma_{G-\{v_1,v_3,v_4\}}$ can be extended to a recoloring sequence $\sigma_{G-\{v_1,v_4\}}$ from $\alpha|_{G-\{v_1,v_4\}}$ to $\beta|_{G-\{v_1,v_4\}}$ that recolors $v_3$ at most $\big\lceil \frac{2\cdot 29+13}{8-3-1}\big\rceil +1=19$ times. Similarly, $\sigma_{G-\{v_1,v_4\}}$ can be extended to $\sigma_{G-v_1}$ and then $\sigma_{G-v_1}$ can be extended to $\sigma_{G}$ from $\alpha$ to $\beta$ that recolors $v_4$ at most $\big\lceil \frac{3\cdot 29}{8-3-1}\big\rceil +1=23$ times and $v_1$ at most $\big\lceil \frac{29+13+19+23}{8-4-1}\big\rceil +1=29$ times. 
\end{proof}

\section{Graphs with bounded independence number}

In this section, we prove Theorem~\ref{boundedindependence}. For a color $i$ of a coloring $c$ of a graph $G$, $c^{-1}(i)$ is the set of vertices of $G$ that are assigned color $i$ in $c$. For two colorings $c_1$ and $c_2$ of $G$, we say that $c_1$ and $c_2$ share a color class if there exist colors $i$ and $j$ of $c_1$ and $c_2$, respectively such that $c_1^{-1}(i)=c_2^{-1}(j)$. A \emph{frozen $k$-coloring} of a graph $G$ is a $k$-coloring of $G$ such that all the $k$ colors appear in $N(v)\cup \{v\}$ for every $v\in V(G)$. Note that a frozen $k$-coloring of a graph $G$ is an isolated vertex in $\mathcal{G}(G,k)$. So to show that $\mathcal{G}(G,k)$  is disconnected, it is sufficient to show the existence of a frozen $k$-coloring of $G$. We divide the proof of Theorem~\ref{boundedindependence} into two parts. In Lemma~\ref{frozen}, we show the existence of a graph $G$ such that $\mathcal{G}(G,\big\lfloor \frac{pk}{2}\big\rfloor)$ is disconnected. In Lemma~\ref{mainlemma}, we show that diam$(\mathcal{G}(G,k'))\leq 4|V(G)|$ for every $k'\geq \big\lfloor \frac{pk}{2}\big\rfloor+1$. So the proof of Theorem~\ref{boundedindependence} follows by Lemma~\ref{frozen} and Lemma~\ref{mainlemma}. 

\begin{figure}[hbtp]
\centering
\includegraphics[scale=.6]{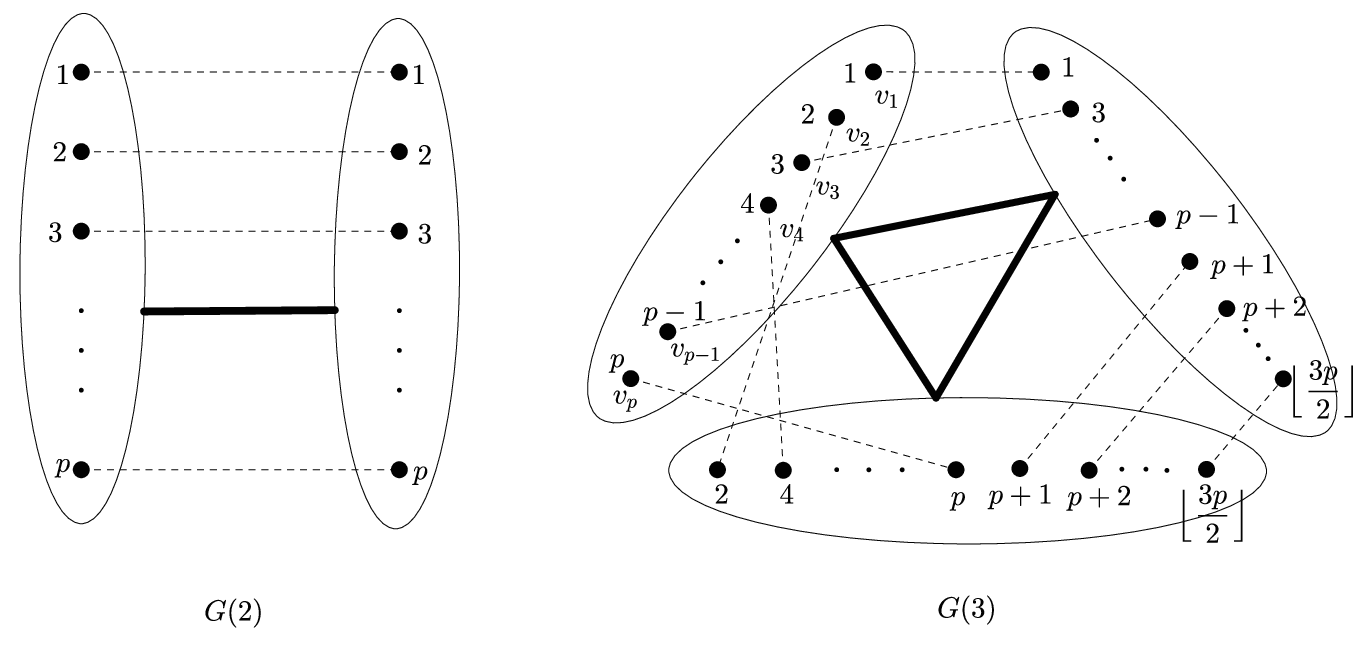}
\caption{Elliptical regions represent independent sets. Bold line between two sets represents that all the edges between the sets are present except the edges represented by dashed lines. The number next to a vertex is the color assigned to that vertex.}\label{figure4}
\end{figure}

\begin{lemma}\label{frozen}
For every $p\geq 2$ and $k\geq 2$, there exists a $k$-colorable graph with the independence number at most $p$ such that $\mathcal{G}(G,\big\lfloor \frac{pk}{2}\big\rfloor)$ is disconnected.
\end{lemma}
\begin{proof}
Fix an integer $p\geq 2$. To prove the lemma, first we show the $k=2$ case by showing the existence of a $2$-colorable graph $G(2)$ with the independence number at most $p$ such that $\mathcal{G}(G(2),p)$ is disconnected. Then we show the $k=3$ case by showing the existence of a $3$-colorable graph $G(3)$ with the independence number at most $p$ such that $\mathcal{G}(G(3),\big\lfloor \frac{3p}{2}\big\rfloor)$ is disconnected. Then by using these two special graphs, we show the existence of a $k$-colorable graph $G(k)$ with the independence number at most $p$ such that $\mathcal{G}(G(k),\big\lfloor \frac{pk}{2}\big\rfloor)$ is disconnected for every $k\geq 4$.

Let $G(2)$ be the bipartite graph obtained by removing a perfect matching from a complete bipartite graph with both the partite sets having cardinality $p$. Note that $G(2)$ has independence number $p$. Now properly color the graph $G(2)$ such that every color from the set $\{1,2,\ldots,p\}$ appears in each partite set of $G(2)$ (see Figure~\ref{figure4}(left)). Note that the coloring obtained is a frozen $p$-coloring of $G(2)$ and hence $\mathcal{G}(G(2),p)$ is disconnected. Since $G(2)$ is $2$-colorable, Lemma~\ref{frozen} holds for $k=2$. The construction of the graph $G(2)$ first appeared in \cite{cerecedaconjec}. 

Now we construct a $3$-colorable graph $G(3)$ with the independence number at most $p$ such that $\mathcal{G}(G(3),\big\lfloor \frac{3p}{2}\big\rfloor)$ is disconnected as follows. If $p$ is even, then $V(G(3))$ has $3p$ vertices; otherwise $V(G(3))$ has $3p-1$ vertices. The vertex set $V(G(3))$ can be partitioned as $I_1\cup I_2 \cup I_3$ such that $I_1, I_2$, and $I_3$ are independent sets and $|I_1|=|I_2|=p$. Let $I_1=\{v_1,v_2,\ldots, v_{p}\}$. There exists a $\big\lfloor \frac{3p}{2}\big\rfloor$-coloring $c$ of $G(3)$ such that the following holds. The vertex $v_i$ is assigned color $i$ for every $i\leq p$. Every vertex of $I_2$ is assigned a distinct color $i$ if $i\leq p$ and $i$ is odd. Every vertex of $I_3$ is assigned a distinct color $i$ if $i\leq p$ and $i$ is even. Every remaining vertex of $I_2$ is assigned a distinct color from the set $\{p+1,p+2,\ldots,\big\lfloor \frac{3p}{2}\big\rfloor\}$. Similarly, every remaining vertex of $I_3$ is assigned a distinct color from the set $\{p+1,p+2,\ldots,\big\lfloor \frac{3p}{2}\big\rfloor\}$. All possible edges are present in $G(3)$ such that $I_1,I_2$, and $I_3$ remain independent sets and $c$ remains a proper coloring. We refer to Figure~\ref{figure4}(right) for a representation of $G(3)$ when $p$ is even. Note that $c$ is a frozen $\big\lfloor\frac{3p}{2}\big\rfloor$-coloring of $G(3)$. Hence Lemma~\ref{frozen} holds for $k=3$.

Now assume that $k\geq 4$. Note that there exist integers $x$ and $y$ such that $k=2x+3y$ and $y\in \{0,1\}$. Consider a graph $G(k)$ that is constructed by taking $x$ copies of $G(2)$ and $y$ copy of $G(3)$ and adding all possible edges between every pair of copies. Note that $G(k)$ has independence number at most $p$. Since $G(2)$ has a frozen $p$-coloring and $G(3)$ has a frozen $\big\lfloor \frac{3p}{2}\big\rfloor$-coloring, $G(k)$ has a frozen $xp+y\big\lfloor \frac{3p}{2}\big\rfloor$-coloring. If $k$ is even, then $x=\frac{k}{2}$ and $y=0$. If $k$ is odd, then $x=\frac{k-3}{2}$ and $y=1$. So for any $k\geq 4$, $G(k)$ has a frozen $\big\lfloor \frac{pk}{2}\big\rfloor$-coloring. Hence $G(k)$ is a $k$-colorable graph with the independence number at most $p$ such that $\mathcal{G}(G(k),\big\lfloor \frac{pk}{2}\big\rfloor)$ is disconnected.
\end{proof}

\begin{lemma}[Renaming Lemma~\cite{bonamytree}]\label{renaminglem}
If $c_1$ and $c_2$ are two $k$-colorings of a graph $G$ that induce the same partition of vertices into color classes, then there exists a recoloring sequence $\sigma_G$ from $c_1$ to $c_2$ in $\mathcal{G}(G,k'); k'>k$ that recolors every vertex at most $2$ times.
\end{lemma}

\begin{lemma}\label{l1}
Let $G$ be a graph with the independence number at most $p$ for some fixed $p\geq 2$. If $G$ is $k$-colorable for some positive integer $k$, then in any $k'$-coloring of $G$ for $k'\geq\big\lfloor \frac{pk}{2}\big\rfloor+1$, there exists a color class that contain at most one vertex of $G$.
\end{lemma}
\begin{proof}
Let $G$ be $k$-colorable. Since $G$ has independence number at most $p$, every color class in any coloring of $G$ contains at most $p$ vertices of $G$. Since $G$ is $k$-colorable, we have $|V(G)|\leq pk$. For the sake of contradiction, assume that there exists a $k'$-coloring of $G$ such that $k'\geq \big\lfloor \frac{pk}{2}\big\rfloor+1$ and every color class contains at least two vertices of $G$. Then $|V(G)|\geq 2k'\geq 2 (\big\lfloor \frac{pk}{2}\big\rfloor+1)>pk$ which is a contradiction. 
\end{proof}

\begin{lemma}\label{l2}
Let $c_1$ be a $k$-coloring of $G$ such that it has a color class with at most one vertex of $G$. Given any $k$-coloring $c$, there exists a $k$-coloring $c_2$ that shares a color class with $c$, say $C$, such that $c_2$ can be obtained from $c_1$ by recoloring each vertex of $C$ at most once.
\end{lemma}

\begin{proof}
Let $i$ be a color of $c_1$ such that $|c_1^{-1}(i)|\leq 1$. Now we choose a color $j$ of $c$ as follows. If $c_1^{-1}(i)=\{v\}$, then $j$ is the color assigned to $v$ in $c$. If $c_1^{-1}(i)=\emptyset$, then $j$ is any color of $c$ whose color class is non-empty. Let $C=c^{-1}(j)$. Recolor the vertices of $C$ by the color $i$ in the coloring $c_1$ one by one. Since the vertices of $C$ are pairwise non-adjacent, at each step a proper $k$-coloring is obtained. Let $c_2$ be the new coloring obtained from $c_1$ at the end of the process. Then it is clear that $c_2^{-1}(i)=C=c^{-1}(j)$ and there is a recoloring sequence from $c_1$ to $c_2$ that recolors every vertex of $C$ at most once. 
\end{proof}

\begin{lemma}\label{mainlemma}
Let $G$ be a graph with the independence number at most $p$ for some fixed $p\geq 2$. If $G$ is $k$-colorable for some $k\geq 2$, then diam$(\mathcal{G}(G,k'))\leq 4|V(G)|$ for every $k'\geq \big\lfloor \frac{pk}{2}\big\rfloor+1$.
\end{lemma}	
\begin{proof}
Let $G$ be $k$-colorable for some $k\geq 2$. Let $c_1$ and $c_2$ be two distinct $k'$-colorings of $G$ for $k'\geq \big\lfloor \frac{pk}{2}\big\rfloor+1$. It is sufficient to show that there exists a recoloring sequence from $c_1$ to $c_2$ that recolors every vertex of $G$ at most $4$ times. Let $\gamma$ be any $\chi(G)$-coloring of $G$. 

\begin{claim}\label{mainclaim}
 There exists a $\chi(G)$-coloring $c^*$ such that $c^*$ and $\gamma$ partition $V(G)$ into the same color classes and $c^*$ can be obtained from $c_1$ by recoloring every vertex of $G$ at most once.
\end{claim}
\begin{proof}[Proof of Claim~\ref{mainclaim}]
 We use induction on $\chi(G)$ to prove the claim. For the base case, assume that $\chi(G)=1$. Then $G$ consists of only isolated vertices and hence $c^*$ can be obtained from $c_1$ by recoloring every vertex of $G$ by a single color. So assume that $\chi(G)\geq 2$ and the induction hypothesis that the claim is true for all the graphs whose chromatic number is less than $\chi(G)$. Since $G$ is $k$-colorable and $c_1$ is a $k'$-coloring of $G$ for $k'\geq \big\lfloor \frac{p k}{2}\big\rfloor+1$, by Lemma~\ref{l1}, there exists a color class of $c_1$ that contains at most one vertex of $G$. 
 Then by Lemma~\ref{l2}, a $k'$-coloring $c_1'$ of $G$  exists such that $c_1'$ and $\gamma$ share at least one color class, say $C$, and $c_1'$ can be obtained from $c_1$ by recoloring the vertices of $C$ at most once. Let $i$ and $j$ be the colors of $c_1'$ and $\gamma$, respectively such that $c_1'^{-1}(i)=\gamma^{-1}(j)=C$. Without loss of generality, we may assume that $i=k'$ and $j=\chi(G)$. 

Let $G_1$ be the subgraph of $G$ induced by the vertex set $V(G)\setminus C$. Let $\hat{c_1}$ and $\hat{\gamma}$ be the $(k'-1)$-coloring $c_1'|_{G_1}$ and the $(\chi(G)-1)$-coloring $\gamma|_{G_1}$, respectively. Note that $\hat{\gamma}$ is a $\chi(G_1)$-coloring of $G_1$ since $\chi(G_1)=\chi(G)-1$. Let $k'_1=k'-1$ and $k_1=k-1$. Then $\hat{c_1}$ and $\hat{\gamma}$ are $k'_1$-coloring and $\chi(G_1)$-coloring of $G$, respectively. Note that $k'_1=k'-1\geq (\big\lfloor \frac{pk}{2}\big\rfloor+1)-1\geq \big\lfloor \frac{p(k-1)}{2}\big\rfloor+1=\big\lfloor \frac{p k_1}{2}\big\rfloor+1$.
Further note that $G_1$ is $k_1$-colorable since $G$ is $k$-colorable. So by induction, there exists a $k_1'$-coloring $c_1^*$ of $G_1$ such that $c_1^*$ and $\hat{\gamma}$ partition $V(G_1)$ into same sets of color classes and $c_1^*$ can be obtained from $\hat{c_1}$ by recoloring every vertex of $G_1$ at most once. Let $c^*$ be a $k'$-coloring of $G$ such that $c^*|_{G_1}=c_1^*$ and $c^*(v)=c_1'(v)$ for $v\in C$. Thus there is a recoloring sequence from $c_1'$ to $c^*$ that recolors the vertices of $G_1$ at most once and does not recolor the vertices of $C$. So there is a recoloring sequence from $c_1$ to $c^*$ that recolors every vertex of $G$ at most once and $c^*$ and $\gamma$ partition $V(G)$ into the same color classes. Note that the first $k'-\chi(G)$ color classes of $c^*$ are empty and hence $c^*$ uses $\chi(G)$ colors only.  This completes the proof of the claim.
\end{proof}

Now we return to the proof of Lemma~\ref{mainlemma}. By Claim~\ref{mainclaim}, there exist $\chi(G)$-colorings $c^*$ and $c^{**}$ such that there are recoloring sequences from $c_1$ to $c^*$ and from $c_2$ to $c^{**}$ that recolors every vertex of $G$ at most once. Recall that $c^*$ and $c^{**}$ are $\chi(G)$-colorings of $G$ such that $c^*$, $c^{**}$, and $\gamma$ induce the same partition of vertices into color classes. Since $k'>\chi(G)$, by Lemma~\ref{renaminglem}, there is a recoloring sequence from $c^*$ to $c^{**}$ in $\mathcal{G}(G,k')$ that recolors every vertex of $G$ at most twice. Hence there is a recoloring sequence from $c_1$ to $c_2$ that recolors every vertex of $G$ at most $4$ times.
\end{proof}	

\end{document}